\documentclass{amsproc}[11pt]

\usepackage[ansinew]{inputenc}
\usepackage{graphicx}
\usepackage{color}

\definecolor{rr}{rgb}{.8,0,.3}

\usepackage{enumerate,latexsym}
\usepackage{latexsym}
\usepackage{amsmath,amssymb}
\usepackage{graphicx}
\newfont{\bb}{msbm10 at 11pt}
\newfont{\bbsmall}{msbm8 at 8pt}

\def\rth{\mathbb{R}^3}
\def\R{\mathbb{R}}
\def\B{\mathbb{B}}

\def\Z{\mathbb{Z}}
\def\C{\mathbb{C}}
\def\Q{\mathbb{Q}}

\def\esf{\mathbb{S}}

\newcommand{\la}{\looparrowright}

\newcommand{\ben}{\begin{enumerate}}
\newcommand{\bit}{\begin{itemize}}
\newcommand{\een}{\end{enumerate}}
\newcommand{\eit}{\end{itemize}}
\newcommand{\wh}{\widehat}

\newcommand{\wt}{\widetilde}

\newcommand{\sol}{\mathrm{Sol}_3}
\newcommand{\ed}{\end{document}}

\def\a{{\alpha}}

\def\G{{\Gamma}}
\def\l{{\lambda}}

\def\de{{\delta}}

\def\ve{{\varepsilon}}

\def\cA{\mathcal{A}}

\def\cC{\mathcal{C}}

\def\cH{\mathcal{H}}
\def\cL{\mathcal{L}}
\def\cM{\mathcal{M}}

\newcommand{\chM}{\mathcal{M}^1}
\def\cG{\mathcal{G}}

\def\cT{\mathcal{T}}

\let\8=\infty \let\0=\emptyset  
\let\hat=\widehat

\let\landa=\lambda
\let\alfa=\alpha

\def\landa{\lambda}

\def\flecha{\rightarrow}
\def\esiz{\langle}
\def\esde{\rangle}

\def\cte.{\mathop{\rm cte.}\nolimits}

\def\det{\mathop{\rm det}\nolimits}

 \def\dim{\mathop{\rm dim }\nolimits}

\def\E{\mathbb{E}}
\def\B{\mathbb{B}}
\def\Q{\mathbb{Q}}
\def\R{\mathbb{R}}
\def\Z{\mathbb{Z}}
\def\C{\mathbb{C}}

\def\H{\mathbb{H}}
\def\S{\mathbb{S}}
\def\P{\mathbb{P}}

\def\cb{\overline{\mathbb{C}}}

\newtheorem{theorem}{Theorem}[section]
\newtheorem{lemma}[theorem]{Lemma}
\newtheorem{proposition}[theorem]{Proposition}

\newtheorem{remark}[theorem]{Remark}
\newtheorem{corollary}[theorem]{Corollary}
\newtheorem{definition}[theorem]{Definition}

\newtheorem{fact}[theorem]{Fact}

\newcommand{\su}{{\rm SU}(2)}
\newcommand{\EE}{\wt{\mathrm E}(2)}

\newtheorem{ddescription}[theorem]{Description}

\renewcommand{\sl}{\wt{\mathrm{SL}}(2,\R)}

\numberwithin{equation}{section}

\begin{document}

\begin{title}
{Constant mean curvature spheres in  homogeneous three-spheres}
\end{title}
\today
\author{William H. Meeks III}
\address{William H. Meeks III, Mathematics Department,
University of Massachusetts, Amherst, MA 01003}
\email{profmeeks@gmail.com}
\thanks{The first author was supported in part by NSF Grant DMS -
   1004003. Any opinions, findings, and conclusions or recommendations
   expressed in this publication are those of the authors and do not
   necessarily reflect the views of the NSF}
\author{Pablo Mira}
\address{Pablo Mira, Department of
Applied Mathematics and Statistics, Universidad Polit´ecnica de
Cartagena, E-30203 Cartagena, Murcia, Spain.}

\email{pablo.mira@upct.es}
\thanks{The second author was supported in part by
MICINN-FEDER, Grant No. MTM2010- 19821 and Fundación Séneca, R.A.S.T
2007-2010, reference 04540/GERM/06.}

\author{Joaqu\'\i n P\'erez}
\address{Joaqu\'\i n P\'erez, Department of Geometry and Topology,
University of Granada, 18001 Granada, Spain}
 \email{jperez@ugr.es}

\author{Antonio Ros}
\address{Antonio Ros, Department of Geometry and Topology,
University of Granada, 18001 Granada, Spain}
 \email{aros@ugr.es}
\thanks{The third and fourth authors were supported in part
by MEC/FEDER grants no. MTM2007-61775 and MTM2011-22547, and
Regional J. Andaluc\'\i a grant no. P06-FQM-01642}





\begin{abstract}
We  give a complete classification of the immersed constant mean
curvature spheres in a three-sphere with an arbitrary homogenous
metric, by proving that for each $H\in\R$, there exists a constant
mean curvature $H$ sphere in the space that is unique up to an
ambient isometry.
\end{abstract}
\maketitle

\section{Introduction.}
\label{sec:introduction} The classification and geometric
description of constant mean curvature spheres is a fundamental
problem in classical surface theory. There are two highly
influential results on this problem, see \cite{AbRo1,AbRo2,hf1}:

\begin{theorem}[Hopf Theorem] \label{hopf}
An immersed sphere of constant mean curvature in a complete, simply
connected three-dimensional manifold $\Q^3(c)$ of constant sectional
curvature $c$ is a round sphere.
 \end{theorem}

\begin{theorem}[Abresch-Rosenberg Theorem] \label{abro}
An immersed sphere of constant mean curvature in a simply connected
homogeneous three-manifold with a four-dimensional isometry group is
a rotational sphere.
\end{theorem}
When the ambient space is an arbitrary homogeneous three-manifold
$X$, the type of description of constant mean curvature spheres
given by the above theorems is no longer possible due to the lack of
ambient rotations in $X$. Because of this, the natural way to
describe constant mean curvature spheres in that general setting is
to parameterize explicitly the moduli space of these spheres, and to
determine their most important geometric properties.

In this manuscript we develop a theoretical framework for  studying
constant mean curvature surfaces in  any simply connected
homogeneous three-manifold $X$ that is not diffeomorphic to
$\esf^2\times \R $ (the classification of constant mean curvature
spheres in $\esf^2\times \R $ follows from Theorem~\ref{abro}). We
will apply this general theory to the classification and geometric
study of constant mean curvature spheres when $X$ is compact.
Specifically, Theorem \ref{main} below gives a classification of
constant mean curvature spheres in any homogeneous three-manifold
diffeomorphic to $\S^3$, by explicitly determining the essential
properties of such spheres with respect to their existence,
uniqueness, moduli space, symmetries, embeddedness and stability.

\begin{theorem}
\label{main} Let $X$ be a compact, simply connected homogenous
three-manifold. Then:
\ben
\item \label{exist} For every $H\in \R$,
there exists an immersed oriented sphere $S_H$ in $X$ of constant
mean curvature $H$.
\item \label{uniq} Up to ambient isometry, $S_H$
is the unique immersed sphere in $X$ with constant mean curvature
$H$.
\item \label{familySt} There exists
 a well-defined point in $X$ called
the  \emph{center of symmetry} of $S_H$
 such that the isometries of $X$ that fix this point also leave
 $S_H$  invariant.
 \item \label{alex-embed}
$S_H$ is Alexandrov embedded, in the sense that the immersion
$f\colon S_H\looparrowright X$ of $S_H$ in $X$ can be extended to an
isometric immersion $F\colon B\to X$ of a Riemannian three-ball
 such that $\partial B=S_H$ is mean convex.
\item \label{index1}
$S_H$ has index one and nullity three for the Jacobi operator.
\end{enumerate}
Moreover, let $\mathcal{M}_X$ be the set of oriented immersed
spheres of constant mean curvature in $X$ whose center of symmetry
is a given point $e\in X$. Then, ${\mathcal M}_X$ is an analytic
family $\{ S(t) \mid t \in \R\} $ parameterized by the mean
curvature value $t$ of $S(t)$.
\end{theorem}

Every compact,  simply connected homogeneous three-manifold is
isometric to the Lie group $\su =\{ A\in \mathcal{M}_2(\C )\ | \
A^t\overline{A}=I_2, \ \det (A)=1\} $ endowed with a left invariant
metric. There exists a three-dimensional family of such homogeneous
manifolds, which includes the three-spheres $\S^3(c)$ of constant
sectional curvature $c>0$ and the two-dimensional family of
rotationally symmetric \emph{Berger spheres}, which have a
four-dimensional isometry group. Apart from these two more symmetric
families, any other  left invariant metric on $\su$ has a
three-dimensional isometry group, with the isotropy group of any
point being isomorphic to $\Z_2\times\Z_2$. Item~(\ref{familySt}) in
Theorem~\ref{main} provides the natural generalization of the
theorems by Hopf and Abresch-Rosenberg to this more general context,
since it implies that any constant mean curvature sphere $S_H$ in
such a space inherits all the ambient isometries fixing some point;
in particular, $S_H$ is round (resp. rotationally invariant) in
$\S^3(c)$ (resp. in Berger spheres).

Items~(\ref{exist}) and (\ref{uniq}) together with the last
statement of Theorem~\ref{main} provide an explicit description of
the moduli space of constant mean curvatures spheres in any compact,
simply connected homogeneous three-manifold $X$.
Items~(\ref{alex-embed}) and (\ref{index1}) in Theorem~\ref{main}
describe general embeddedness and stability type properties of
constant mean curvature spheres in $X$ which are essentially sharp,
as we explain next. In the manifolds $\S^3(c)$, the constant mean
curvature spheres are round; hence, they are embedded and weakly
stable (see Section~\ref{sec:index1} for the definition of weak
stability). However, for a general homogeneous $X$ diffeomorphic to
$\S^3$, the constant mean curvature spheres need not be embedded
(Torralbo~\cite{tor1} for certain ambient Berger spheres) or weakly
stable (Torralbo and Urbano~\cite{tou1} for certain ambient Berger
spheres, see also Souam~\cite{so3}), and they are not geodesic
spheres if $X$ is not isometric to some  $\S^3(c)$.  Nonetheless,
item~(\ref{alex-embed}) in Theorem~\ref{main} shows that any
constant mean curvature sphere in a general $X$ satisfies a weaker
notion of embeddedness (which is usually  called Alexandrov
embeddedness), while item~(\ref{index1}) describes the index and the
dimension of the kernel of the Jacobi operator of a constant mean
curvature sphere.

An important property of constant mean curvature spheres in $X$ not
listed in the statement of Theorem~\ref{main} is that, after
identifying $X$ with the Lie group $\su$ endowed with a left
invariant metric, the left invariant Gauss map of every constant
mean curvature sphere  in $X$ is a diffeomorphism to $\S^2$; see
Theorem~\ref{thm:index1} for this result.  This diffeomorphism
property will be crucial in the proof of Theorem~\ref{main}.

As an application of Theorem~\ref{main}, we provide in
Theorem~\ref{embed:su2} a more detailed description of the special
geometry of  minimal spheres in a general compact $X$.
In particular, it follows from Theorem~\ref{embed:su2} that minimal spheres
in $X$ are embedded; this follows from
the existence of an embedded minimal sphere in $\S^3$ endowed with
an arbitrary Riemannian metric as proved by \cite{smith1} (see also Remark~\ref{smith}), and the
fact that any two such minimal spheres in $X$ are congruent by
item~(\ref{uniq}) of Theorem~\ref{main}.

The theorems of Hopf and Abresch-Rosenberg rely on the existence of
a holomorphic quadratic differential for constant mean curvature
surfaces in homogeneous three-manifolds with isometry group of
dimension at least four. This approach using holomorphic quadratic
differentials does not seem to work when the isometry group of the
homogeneous three-manifold has dimension three. Instead, our
approach to proving Theorem~\ref{main} is inspired by two recent
works on constant mean curvature spheres in the Thurston geometry
$\sol$, i.e., in the solvable Lie group Sol$_3$ with any of its most
symmetric left invariant metrics. One of these works is the local
parameterization by Daniel and Mira~\cite{dm2} of the space
$\mathcal{M}_{\mbox{\footnotesize Sol}_3}^1$ of index-one, constant
mean curvature spheres in $\sol$ equipped with its standard metric
via the left invariant Gauss map and the uniqueness of such spheres.
The other one is the obtention by Meeks~\cite{me34} of area
estimates for the subfamily of spheres in
$\mathcal{M}_{\mbox{\footnotesize Sol}_3}^1$ whose mean curvatures
are bounded from below by any fixed positive constant; these two
results lead to a complete description of the spheres of constant
mean curvature in $\sol$ endowed with its standard metric. However,
the proof of Theorem~\ref{main} for a general compact, simply
connected homogeneous three-manifold $X$ requires the development of
new techniques and theory.  These new techniques and theory are
needed to prove that the left invariant Gauss map of an index-one
sphere of constant mean curvature in $X$ is a diffeomorphism, that
constant mean curvature spheres in $X$ are Alexandrov embedded and
have a center of symmetry, and that there exist a priori area
estimates for the family of index-one spheres of constant mean
curvature in $X$. On the other hand, some of the arguments in the
proofs of Theorem \ref{rep} and Theorem \ref{thm:index1} below are
generalizations of ideas in Daniel and Mira~\cite{dm2}, and are
therefore merely sketched here. For more specific details on these
computations, we refer the reader to the announcement of the results
of the present paper in Chapter 3 of the Lecture Notes~\cite{mpe11}
by the first and third authors.

Here is a brief outline of the proof of Theorem~\ref{main}. First we
identify the compact, simply connected homogenous three-manifold $X$
isometrically with $\su$ endowed with a left invariant metric. Next
we show that any index-one sphere $S_H$ of constant mean curvature
$H$ in $X$ has the property that any other immersed sphere of
constant mean curvature $H$ in $X$ is a left translation of $S_H$.
The next step in the proof is to show that the set $\cH$ of values
$H\in \R$ for which there exists an index-one sphere of constant
mean curvature $H$ in $X$ is non-empty, open and closed in $\R$
(hence, $\cH=\R$). That $\cH$ is non-empty follows from the
existence of small isoperimetric spheres in $X$. Openness of $\cH$
follows from an argument using the implicit function theorem, which
also proves that the space of index-one spheres with constant mean
curvature in $X$ modulo left translations is an analytic
one-dimensional manifold. By elliptic theory, closedness of $\cH$
can be reduced to obtaining a priori area and curvature estimates
for index-one, constant mean curvature spheres with any fixed upper
bound on their mean curvatures. The existence of these curvature
estimates is obtained by a rescaling argument. The most delicate
part of the proof of Theorem~\ref{main} is obtaining a priori  area
estimates. For this, we first show that the non-existence of an
upper bound on the areas of all constant mean curvature spheres in
$X$ implies the existence of a complete, stable, constant mean
curvature surface in $X$ which can be seen to be the lift via a
certain fibration $\Pi\colon X\to \S^2$  of an immersed curve in
$\S^2$, and then we prove that such a surface cannot be stable to
obtain a contradiction. The Alexandrov embeddedness of constant mean
curvature spheres follows from a deformation argument, using the
smoothness of the family of constant mean curvature spheres in $X$.
Finally, the existence of a center of symmetry for any constant mean
curvature sphere in $X$ is deduced from the Alexandrov embeddedness
and the uniqueness up to left translations of the sphere.

Even though the geometry of constant mean curvature surfaces in
homogeneous three-manifolds with an isometry group of dimension at
least four has been deeply studied, the case where the ambient space
is an arbitrary homogeneous three-manifold remains largely
unexplored. The results in this paper and the Lecture
Notes~\cite{mpe11}  seem to constitute the first systematic study of
constant mean curvature surfaces in general homogeneous
three-manifolds, as well as  a starting point for further
development of this area. It is worth mentioning that many of the
results in this paper are proven for any homogeneous, simply
connected three-manifold not diffeomorphic to $\esf^2\times \R$. In
our forthcoming paper~\cite{mmpr1} these results provide the
foundation for understanding the geometry and classification of
constant mean curvature spheres in any homogeneous three-manifold
diffeomorphic to $\R^3$.

\section{Background material on three-dimensional metric Lie groups.}
 \label{sec:background}
 We next state some basic properties of
 three-dimensional Lie groups endowed
with a left invariant metric that will be used freely in later
sections. For details of these basic properties, see Chapter 2 of
the general reference~\cite{mpe11}.

Let $Y$ denote a simply connected, homogeneous Riemannian
three-manifold, and  assume that it is not isometric to the
Riemannian product  $\S^2(\kappa)\times \R$, where $\kappa>0$ is the
Gaussian curvature of $\S^2(\kappa)$. Then $Y$ is isometric to a
simply connected, three-dimensional Lie group $G$ equipped with a
left invariant metric $\langle ,\rangle $, i.e., for every $p\in G$,
the left translation $l_p\colon G\to G$, $l_p(q)=p\, q$, is an
isometry of $\langle ,\rangle $. We will call such a space a
\emph{metric Lie group}, and denote it by $X$. For the underlying
Lie group structure of such an $X$, there are two possibilities,
unimodular and non-unimodular.

\subsection{$X$ is unimodular.}
\label{secunimodgr} Among all simply connected, three-dimensional
Lie groups, the cases $\su$, $\sl$ (universal cover of
$\mathrm{SL}(2,\R) $), $\EE $ (universal cover of the Euclidean
group of orientation preserving rigid motions of the plane), $\sol$
(Sol geometry, or the universal cover of the group of orientation
preserving rigid motions of the Lorentzian plane), Nil$_3$
(Heisenberg group of upper triangular $3\times3$ real matrices) and
$\R^3$ comprise the \emph{unimodular} Lie groups.

Suppose that $X$ is a simply connected, three-dimensional unimodular
Lie group  equipped with a left invariant metric $\langle ,\rangle
$. It is always possible to find an orthonormal left invariant frame
$\{E_1,E_2,E_3\}$ such that
\begin{equation}
\label{eq:11}
 [E_2,E_3]=c_1E_1,\quad [E_3,E_1]=c_2E_2, \quad
[E_1,E_2]=c_3E_3,
\end{equation}
for certain constants $c_1,c_2,c_3\in \R$, among which  at most one
$c_i$  is negative. The triple of numbers $(c_1,c_2,c_3)$ and
$\{E_1,E_2,E_3\}$ are called   the {\it structure constants}  and
the \emph{canonical basis} of the unimodular metric Lie group $X$,
respectively. The following associated constants are also useful in
describing the geometry of a simply connected, three-dimensional
unimodular metric Lie group:

\begin{equation} \label{def:mu}
\mu _1=\frac{1}{2}(-c_1+c_2+ c_3),\quad  \mu _2=\frac{1}{2}(c_1-c_2+
c_3),\quad  \mu _3=\frac{1}{2}(c_1+c_2-c_3).
\end{equation}
For instance, the Levi-Civita connection associated to $\langle,
\rangle $ is given by
\begin{equation}
\label{eq:LCunim} \nabla _{E_i}E_j=\mu _i\, E_i\times E_j,\quad
i,j\in \{ 1,2,3\} ,
\end{equation}
where $\times $ denotes the cross product associated to $\langle
,\rangle $ and to the orientation on $X$ defined by declaring
$(E_1,E_2,E_3)$ to be a positively oriented basis. From here it is
easy to check that the symmetric Ricci tensor associated to $\langle
,\rangle $ diagonalizes in the basis $\{ E_1,E_2,E_3\} $ with
eigenvalues (see Milnor~\cite{mil2} for details)
\begin{equation}
\label{Ricunimod}
 \mbox{Ric}(E_1)=2\mu _2\mu _3, \quad
\mbox{Ric}(E_2)=2\mu _1\mu _3,
 \quad \mbox{Ric}(E_3)=2\mu _1\mu _2.
\end{equation}

 Depending on the signs of the
structure constants $c_i$, we obtain six possible different Lie
group structures, which are listed  in the table below together with
the possible dimension of the isometry group $\mbox{Iso}(X)$ for a
given left invariant metric on the Lie group $X$.
 \par
\vspace{.1cm}
 \begin{center}

 \begin{tabular}{|c|c|c|c|}
        \hline
        Signs of $c_1,c_2,c_3$ & $\dim \mbox{Iso}(X)=3$ & $\dim
        \mbox{Iso}(X)=4$
        & $\dim \mbox{Iso}(X)=6$ \\ \hline
\rule{0cm}{.4cm}
        +,\ +,\ + & $\su$ & $\esf^3_{\mbox{\tiny Berger}}
        =\E (\kappa >0,\tau ) $ & $\esf^3(\kappa )$ \\
\rule{0cm}{.4cm}
        +,\ +,\ -- & $\sl $ & $\E (\kappa <0,\tau )$ & \mbox{\O }\\
\rule{0cm}{.4cm}
        +,\ +,\ 0 & $\widetilde{\mbox{\rm E}}(2)$ & \mbox{\O }
        & $(\widetilde{\mbox{\rm E}}(2),\mbox{flat})$ \\
\rule{0cm}{.4cm}
        +,\ --,\ 0 & Sol$_3$ & \mbox{\O } & \mbox{\O }\\
\rule{0cm}{.4cm}
        +,\ 0,\ 0 & \mbox{\O } & Nil$_3=\E (0,\tau )$ & \mbox{\O }  \\
\rule{0cm}{.4cm}
        0,\ 0,\ 0 & \mbox{\O } & \mbox{\O } & $\R^3$ \\
        \hline
      \end{tabular}
      \\ \mbox{}
\end{center}
{\sc Table 1.} Three-dimensional, simply connected unimodular metric
Lie groups. Each horizontal line corresponds to a unique Lie group
structure; when all the structure constants are different, the
isometry group of $X$ is three-dimensional. If  two  or more
constants agree, then the isometry group of $X$ has dimension $4$ or
$6$. We have used the standard notation $\E (\kappa ,\tau )$ for the
total space of the Riemannian submersion with bundle curvature $\tau
$ over a complete simply connected surface of constant curvature
$\kappa $.
\par
\vspace{.2cm} Fix a unimodular metric Lie group $X$ with underlying
Lie group $Z$ and left invariant metric $\langle ,\rangle $. Let
$E_1,E_2,E_3$ be a ($\langle ,\rangle $-orthonormal) canonical basis
of $X$. Now pick numbers $\l _1,\l _2,\l _3>0$ and declare the
length of $E_i$ to be $\l _i$, while keeping them orthogonal. This
defines uniquely a left invariant metric $\langle ,\rangle '$ on
$Z$, and every left invariant metric on $Z$ can be described with
this procedure (see e.g., the discussion following Corollary~4.4
in~\cite{mil2}).
It turns out that $E_1,E_2,E_3$ are Ricci eigendirections for every
$\langle ,\rangle '$, although the associated Ricci eigenvalues
depend on the $\l _i$ and so, on $\langle ,\rangle '$.

For $i=1,2,3$, the integral curve $\G _i$ of $E_i$ passing through
the identity element $e\in Z$ is a 1-parameter subgroup of $Z$.

\begin{proposition}
\label{propos3.2}
 In the above situation, each $\G _i$ is the fixed
point set of an order-two, orientation preserving isomorphism $\phi
_i\colon Z\to Z$ which we will call the {\em rotation of angle $\pi
$ about $\G _i$}. Moreover, the following properties hold:
\begin{enumerate}[(1)]
\item $\phi _i$ leaves invariant each of the collections of left cosets of
$\G_j$, $j=1,2,3$.
\item For every left invariant metric $\langle ,\rangle '$ of $Z$,
$\G _i$ is a geodesic and $\phi _i$ is an isometry.
\end{enumerate}
\end{proposition}
\begin{proof}
In Proposition 2.21 of~\cite{mpe11} it is proved that, in the above
conditions, there exist order two, orientation-preserving
isomorphisms $\phi _i\colon Z\flecha Z$, $i=1,2,3$, such that
$(d\phi _i)_e((E_i)_e)=(E_i)_e$, and such that the pulled-back
metric  $\phi _i^*\langle ,\rangle '$ equals $\langle ,\rangle '$
for every left invariant metric $\langle ,\rangle '$ on $Z$.
Equation (\ref{eq:LCunim}) implies that the integral curve $\G _i$
of $E_i$ is a geodesic of $(Z,\langle ,\rangle ')$ (for every left
invariant metric $\langle ,\rangle '$). Since $(d\phi
_i)_e((E_i)_e)=(E_i)_e$ , we conclude from uniqueness of geodesics
that $\G_i$ consists entirely of fixed points of $\phi _i$. This
proves Proposition~\ref{propos3.2}.
\end{proof}

\subsection{$X$ is non-unimodular.}
The simply connected, three-dimensional, non-unimodular metric Lie
groups correspond to semi-direct products $X=\R^2\rtimes_A \R$ with
${\rm trace} (A)\neq 0$, as we explain next. A semi-direct product
$\R^2\rtimes_A \R$ is the Lie group $(\R^3\equiv \R^2\times \R,*)$,
where the group operation $*$ is expressed in terms of some real
$2\times2$ matrix $A\in \cM_2(\R)$ as
\begin{equation}
\label{eq:5}
 ({\bf p}_1,z_1)*({\bf p}_2,z_2)=({\bf p}_1+ e^{z_1 A}\  {\bf
 p}_2,z_1+z_2);
\end{equation}
here $e^B=\sum _{k=0}^{\infty }\frac{1}{k!}B^k$ denotes the usual
exponentiation of a matrix $B\in \cM_2(\R )$. Let

\begin{equation} \label{equationgenA} A=\left(
\begin{array}{cr}
a & b \\
c & d \end{array}\right) .
\end{equation}
Then, a left invariant frame $\{ E_1,E_2,E_3\} $ of $X$ is given by
\begin{equation}
\label{eq:6*}
 E_1(x,y,z)=a_{11}(z)\partial _x+a_{21}(z)\partial _y,\quad
E_2(x,y,z)=a_{12}(z)\partial _x+a_{22}(z)\partial _y,\quad
 E_3=\partial _z,
\end{equation}
where
\begin{equation}
\label{eq:exp(zA)}
 e^{zA}=\left(
\begin{array}{cr}
a_{11}(z) & a_{12}(z) \\
a_{21}(z) & a_{22}(z)
\end{array}\right) .
\end{equation}
In terms of $A$, the Lie bracket relations are:
\begin{equation}
\label{eq:8a} [E_1,E_2]=0, \quad [E_3,E_1]=aE_1+cE_2, \quad
 [E_3,E_2]=bE_1+dE_2.
 \end{equation}

\begin{definition}
 \label{def2.1}
 {\rm
We define the {\it canonical left invariant metric} on the
semidirect product $\R^2\rtimes _A\R $ to be that one for which the
left invariant basis $\{ E_1,E_2,E_3\} $ given by (\ref{eq:6*}) is
orthonormal. Equivalently, it is the left invariant extension to
$X=\R^2\rtimes _A\R $ of the inner product on the tangent space
$T_eX$ at the identity element $e=(0,0,0)$ that makes $(\partial
_x)_e,(\partial _y)_e,(\partial _z)_e$ an orthonormal basis.}
\end{definition}

We next emphasize some other metric properties of the canonical left
invariant metric $\langle ,\rangle $ on $\R^2\rtimes_A \R$:

\begin{enumerate}[$\bullet $]
\item The mean curvature of each leaf of the foliation $\mathcal{F}=
\{ \R^2\rtimes _A\{ z\} \mid z\in \R \}$ with respect to the unit
normal vector field $E_3$ is the constant $H=\mbox{trace}(A)/2$. All
the leaves of the foliation $\mathcal{F}$ are intrinsically flat.
\item The change from the orthonormal basis $\{ E_1,E_2,E_3\} $ to the
basis $\{ \partial _x,\partial _y,\partial _z\} $ given by
(\ref{eq:6*}) produces the following expression for the metric
$\langle ,\rangle $ in the $x,y,z$ coordinates of $X$:
\begin{equation}
\label{eq:13}
 \left.
\begin{array}{rcl}
\langle ,\rangle & =&
 \left[ a_{11}(-z)^2+a_{21}(-z)^2\right] dx^2+
\left[ a_{12}(-z)^2+a_{22}(-z)^2\right] dy^2 +dz^2 \\
& + & \rule{0cm}{.5cm} \left[
a_{11}(-z)a_{12}(-z)+a_{21}(-z)a_{22}(-z)\right] \left( dx\otimes
dy+dy\otimes dx\right)
\\
&=& \rule{0cm}{.5cm} e^{-2\mbox{\footnotesize trace}(A)z} \left\{
 \left[ a_{21}(z)^2+a_{22}(z)^2\right] dx^2+
\left[ a_{11}(z)^2+a_{12}(z)^2\right] dy^2\right\} +dz^2 \\
& - & \rule{0cm}{.5cm}
 e^{-2\mbox{\footnotesize trace}(A)z} \left[
a_{11}(z)a_{21}(z)+a_{12}(z)a_{22}(z)\right] \left( dx\otimes
dy+dy\otimes dx\right) .
\end{array}
\right.
\end{equation}

\item The Levi-Civita connection associated to the canonical left invariant
metric is easily deduced from~(\ref{eq:8a}) as follows: {\large
\begin{equation}
\label{eq:12}
\begin{array}{l|l|l}
\rule{0cm}{.5cm} \nabla _{E_1}E_1=a\, E_3 & \nabla
_{E_1}E_2=\frac{b+c}{2}\, E_3
& \nabla _{E_1}E_3=-a\, E_1-\frac{b+c}{2}\, E_2 \\
\rule{0cm}{.5cm} \nabla _{E_2}E_1=\frac{b+c}{2}\, E_3 & \nabla
_{E_2}E_2=d\, E_3
& \nabla _{E_2}E_3=-\frac{b+c}{2}\, E_1-d\, E_2 \\
\rule{0cm}{.5cm} \nabla _{E_3}E_1=\frac{c-b}{2}\, E_2 & \nabla
_{E_3}E_2=\frac{b-c}{2}\, E_1 & \nabla _{E_3}E_3=0.
\end{array}
\end{equation}
}
\end{enumerate}

A simply connected, three-dimensional Lie group is non-unimodular if
and only if it is isomorphic to some semi-direct product
$\R^2\rtimes_A \R$ with ${\rm trace} (A)\neq 0$. If $X$ is a
non-unimodular metric Lie group, then up to the rescaling of the
metric of $X$, we may assume that ${\rm trace} (A)=2$. {\it This
normalization in the non-unimodular case will be assumed throughout
the paper}. After an orthogonal change of the left invariant frame,
we may express the matrix $A$ uniquely as
\begin{equation}
\label{Axieta} A=A(a,b)= \left(
\begin{array}{cc}
1+a & -(1-a)b\\
(1+a)b & 1-a \end{array}\right), \hspace{1cm} a,b\in [0,\infty).
\end{equation}

The \emph{canonical basis} of the non-unimodular metric Lie group
$X$ is, by definition, the left invariant orthonormal frame
$\{E_1,E_2,E_3\}$ given in \eqref{eq:6*} by the matrix $A$
in~(\ref{Axieta}). In other words, every simply connected,
three-dimensional non-unimodular metric Lie group is isomorphic and
isometric (up to possibly rescaling the metric) to $\R^2\rtimes _A\R
$ with its canonical metric, where $A$ is given by (\ref{Axieta}).

We give two examples. If $A=I_2$ where $I_2$ is the identity matrix,
we get a metric Lie group that we denote by $\H^3$, which is
isometric to the hyperbolic three-space with its standard constant
$-1$ metric and where the underlying Lie group structure is
isomorphic to that of the group of similarities of $\R^2$. If $a=1,
b=0$, we get the product space $\H^2(-4)\times \R$ , where
$\H^2(-4)$ has constant negative curvature $-4$.

Under the assumption that $A\neq I_2$, the determinant of $A$
determines uniquely the Lie group structure. This number is the
\emph{Milnor $D$-invariant} of $X=\R^2\rtimes_A \R$:
\begin{equation}
\label{Dxieta} D=(1-a^2)(1+b^2) ={\rm det} (A).
\end{equation}

Every non-unimodular group $X=\R^2\rtimes _A\R $ admits a $\pi
$-rotation about $\partial _z=E_3$, with similar properties to the
three $\pi $-rotations that appeared in Proposition~\ref{propos3.2}
for the unimodular case. The following statement is an elementary
consequence of equations (\ref{eq:13}) and (\ref{eq:12}).
\begin{proposition}
\label{propos3.5} Let $X=\R^2\rtimes _A\R $ be a non-unimodular
metric Lie group endowed with its canonical metric. Then, the
$z$-axis $\{ 0\}\rtimes _A\R $ is a geodesic and it is the fixed
point set of the order-two isomorphism, orientation preserving
isometry $\phi \colon X\to X$, $\phi (x,y,z)=(-x,-y,z)$. In
particular, given any point $p\in X$, the subgroup of
orientation-preserving isometries of $X$ that fix $p$ contains the
subgroup $\{ 1_X, l_p \circ \phi  \circ l_p^{-1}\} \cong \Z _2$.
\end{proposition}

\begin{remark}
  {\rm In the case that the isometry group of a metric Lie group $X$ has dimension three,
  then the orientation-preserving isometries that fix $e$ reduce to the ones that appear
  in Propositions~\ref{propos3.2} and~\ref{propos3.5}. In other words, the group of
orientation-preserving isometries of $X$ that fix a given point
$p\in X$ is
  either isomorphic to $\Z_2\times \Z _2$ (when $X$ is unimodular), or to $\Z_2$ (when
  $X$ is non-unimodular); see \cite[Proposition~2.21]{mpe11}.
  }
\end{remark}

\section{The left invariant Gauss map of constant mean curvature surfaces in
metric Lie groups.} \label{sec:gauss} Throughout this paper, by an
\emph{$H$-surface} we will mean an immersed, oriented surface of
constant mean curvature surface $H\in \R$. By an \emph{$H$-sphere}
we will mean an $H$-surface diffeomorphic to $\S^2$.

It is classically known that the Gauss map of an $H$-surface
$\Sigma$ in $\R^3$ with $H\neq0$ is a harmonic map into $\S^2$, and
that  $\Sigma$  is determined up to translations by its Gauss map.
We next give an extension of this result to the case of an arbitrary
metric Lie group $X$, after exchanging the classical Gauss map by
the left invariant Gauss map $G$ which we define next.

\begin{definition}
\label{defG} {\rm Given an oriented immersed surface $f\colon \Sigma
\looparrowright X$ with unit normal vector field $N\colon \Sigma \to
TX$ (here $TX$ refers to the tangent bundle of $X$), we define the
{\it left invariant Gauss map} of the immersed surface to be the map
$G\colon \Sigma \to \esf^2\subset T_eX$ that assigns to each $p\in
\Sigma $ the unit tangent vector to $X$ at the identity element $e$
given by $(dl_{f(p)})_e(G(p))=N_p$. }
\end{definition}

Observe that the left invariant Gauss map of a two-dimensional
subgroup of $X$ is constant.

We will prove that, after stereographically projecting the left
invariant Gauss map $G$ of an $H$-surface in $X$ from the south
pole\footnote{Here, the south pole of $\esf^2$ is defined in terms
of the canonical basis $\{E_1,E_2,E_3\}$ of $X$.} of the unit sphere
$\esf^2\subset T_eX$, the resulting function $g\colon \Sigma \to
\overline{\C}=\C \cup \{ \infty \} $ satisfies a conformally
invariant elliptic PDE that can be expressed in terms of the
\emph{$H$-potential} of the space $X$, that we define next:

\begin{definition}
\label{Def:HP-non} {\rm Let $X$ be a \underline{non-unimodular}
metric Lie group. Rescale the metric on $X$ so that $X$ is isometric
and isomorphic to $\R^2\rtimes _A\R $ with its canonical metric,
where $A\in \mathcal{M}_2(\R )$ is given by~(\ref{Axieta}) for
certain constants $a,b \geq 0$. Given $H\in\R$, we define the {\it
$H$-potential} of $X$ to be the map $R\colon \overline{\C}\to
\overline{\C} $ given by
\begin{equation}
\label{eq:potentialnonunim}
 R(q)=H\left( 1+|q|^2\right)^2-(1-|q|^4)-a\left(
q^2-\overline{q}^2\right) -ib \left( 2|q|^2-a\left(
q^2+\overline{q}^2\right) \right) ,
\end{equation}
where $\overline{q}$ denotes the complex conjugate of $q\in
\overline{\C }$.}
\end{definition}
\vspace{.1cm}

\begin{definition}
\label{Def:HP-uni} {\rm Let $X$ be a \underline{unimodular} metric
Lie group with structure constants $c_1,c_2,c_3$ defined by
equation~(\ref{eq:11}) and let $\mu _1,\mu _2,\mu _3\in \R$ be the
related numbers defined in~(\ref{def:mu}) in terms of
$c_1,c_2,c_3$. Given $H\in\R $, we define the {\it $H$-potential} of
$X$  as the map $R\colon \overline{\C}\to \overline{\C} $ given by
\begin{equation}
\label{eq:potentialunim}
 R(q)=H\left(
1+|q|^2\right)^2-\frac{i}{2}\left(
 \mu_2|1+q^2|^2+\mu _1|1-q^2|^2+4\mu _3|q|^2\right) .
\end{equation}
}
\end{definition}
Note that, $R(q)/|q|^4$ has a finite limit as $q\to \8$ (in
particular, $R(q)$ blows up when $q\to \8$). We will say that the
$H$-potential $R$ for $X$ has a zero at $q_0=\8\in \overline{\C}$ if
$\lim_{q\to\8} R(q)/|q|^4=0$. A simple analysis of the zeros of the
$H$-potential $R$ in \eqref{eq:potentialnonunim} and
\eqref{eq:potentialunim} gives the following lemma.

\begin{lemma}
\label{nzp} Let $X$ be a metric Lie group and $H\in \R$. Then, the
$H$-potential for $X$ is everywhere non-zero if and only if:
\begin{enumerate}[(1)]
\item $X$ is isomorphic to $\su$, or
\item $X$ is not isomorphic to $\su$, is unimodular and $H\neq0$, or
\item $X$ is non-unimodular with $D$-invariant $D\leq 1$ and $|H|>1$, or
\item $X$ is non-unimodular with $D$-invariant $D>1$ and $|H|\neq 1$.
\end{enumerate}
\end{lemma}

We use next Lemma \ref{nzp} to prove the following fact, that will
be used later on:

\begin{fact}
\label{fact3} Assume that there exists a compact $H$-surface
$\Sigma$ in $X$. Then, the $H$-potential of $X$ is everywhere
non-zero.
\end{fact}
\begin{proof}
If $X$ is non-unimodular, then the horizontal planes $\R^2\rtimes_A
\{z_0\}\subset X=\R^2\rtimes_A \R$ produce a foliation of $X$ by
leaves of constant mean curvature $1$; recall that the matrix $A$ is
given by (\ref{Axieta}). Applying the usual maximum principle to
$\Sigma $ and to the leaves of such foliation we deduce that
$|H|>1$. So, by Lemma \ref{nzp}, the $H$-potential for $X$ does not
vanish in this case.

If $X$ is unimodular and not isomorphic to $\su$, then $X$ is
diffeomorphic to $\R^3$. Suppose $\hat{\Sigma} $ is a compact
minimal surface in $X$. After a left translation, we can assume that
$e\in \hat{\Sigma} $. Choose $\ve >0$ so that the Lie group
exponential map $\exp $ restricts to the ball $\B (\vec{0},\ve
)\subset T_eX$ of radius $\ve $ centered at $\vec{0}\in T_eX$ as a
diffeomorphism into $B=\exp (\B (\vec{0},\ve ))$. Let $p$ be a point
in $(\hat{\Sigma} \cap B)-\{ e\} $ and denote by $\G =\exp (\R v)$
the 1-parameter subgroup of $X$ generated by the unique $v\in \B
(\vec{0},\ve )-\{ \vec{0}\} $ such that $\exp v=p$. As $\G $ is a
proper arc in $X$ and $\hat{\Sigma} $ is compact, there exists a
largest $t_0\in [1,\infty )$ such that $l_{\G (t_0)}(\hat{\Sigma}
)\cap \hat{\Sigma} \neq \mbox{\O}$, where $\G (t)=\exp (tv)$, $t\in
\R $.  Applying the maximum principle for minimal surfaces to
$\hat{\Sigma}$ and $l_{\G (t_0)}(\hat{\Sigma} )$ we deduce that
$\hat{\Sigma} =l_{\G (t_0)}(\hat{\Sigma} )$. By applying $l_{\G
(t_0)}$ again, we have $\hat{\Sigma} =l_{\G (2t_0)}(\hat{\Sigma} )$,
which contradicts the defining property of $t_0$. Therefore, there
are no compact minimal surfaces in $X$, and so $H\neq 0$. By Lemma
\ref{nzp}, the $H$-potential for $X$ is everywhere non-zero.
Finally, if $X$ is isomorphic to $\su$, the same conclusion holds
again by Lemma \ref{nzp}. This proves Fact \ref{fact3}.
\end{proof}

\begin{remark}
{\rm
  If $X$ is unimodular and  not isomorphic to $\su$, then $X$ is either isometric and
  isomorphic to $\R^2\rtimes _A\R $ for a matrix $A\in \mathcal{M}_2(\R )$ with trace
  zero, or to $\sl $. In both cases, it is not difficult to check that there exists a
  minimal two-dimensional subgroup; hence the argument in the first paragraph of the last
  proof could be also adapted to this situation to prove the non-existence of a compact
  immersed minimal surface in $X$.
}
\end{remark}

\begin{theorem}\label{rep}
Suppose $\Sigma$ is a simply connected Riemann surface with
conformal parameter $z$, $X$ is a metric Lie group, and $H\in\R$.

Let $g\colon\Sigma\flecha \overline{\C}$  be a solution of the
complex elliptic PDE
\begin{equation}\label{gauss}
g_{z\overline{z}} = \frac{R_q}{R} (g) \, g_z g_{\overline{z}} +
\left(\frac{R_{\overline{q}}}{R} -
\frac{\overline{R_q}}{\overline{R}}\right) (g)\, |g_z|^2,
\end{equation}
such that $g_z\neq 0$ everywhere\footnote{By $g_z\neq 0$ we mean
that $g_z(z_0)\neq 0$ if $g(z_0)\in \C$ and that $\lim_{z\to z_0}
(g_z/g^2)(z)\neq 0$ if $g(z_0)=\8$.}, and such that the
$H$-potential $R$ of $X$ does not vanish on $g(\Sigma)$ (for
instance, this happens if $H$ satisfies the conditions of
Lemma~\ref{nzp}). Then, there exists an immersed $H$-surface
$f\colon\Sigma\looparrowright X$, unique up to left translations,
whose Gauss map is $g$.

Conversely, if $g\colon\Sigma\flecha \cb$ is the Gauss map of an
immersed $H$-surface $f\colon\Sigma\looparrowright X$ in a metric
Lie group $X$, and the $H$-potential $R$ of $X$ does not vanish on
$g(\Sigma)$, then $g$ satisfies the equation \eqref{gauss}, and
moreover $g_z\neq 0$ holds everywhere.
\end{theorem}
\begin{proof}
We will only consider the case that $X$ is unimodular, since the
non-unimodular situation can be treated in a similar way.

Given a solution $g\colon\Sigma\flecha \overline{\C}$ to
\eqref{gauss} as in the statement of the theorem, define the
functions $A_1,A_2,A_3\colon\Sigma\flecha \overline{\C}$ by

\begin{equation}
\label{eq:Ai}
 A_1=\frac{\eta }{4}\left(
\overline{g}-\frac{1}{\overline{g}}\right) ,\quad
 A_2=\frac{i\eta
}{4}\left( \overline{g}+\frac{1}{\overline{g}}\right) ,\quad
 A_3=\frac{\eta }{2},
\quad \mbox{where \ } \eta =\frac{4\overline{g}g_z}{R(g)},
\end{equation}
and $z$ is a conformal parameter on $\Sigma$. Noting that
$R(q)/|q|^4$ has a finite limit as $q\to \8$, we can easily deduce
from \eqref{eq:Ai} that $A_1,A_2,A_3$ actually take values in $\C$.

A direct computation shows that
\[
(A_1)_{\overline{z}}=\frac{g_{z\overline{z}}(\overline{g}^2-1)}{R(g)}-\frac{g_z
\left[ R(g)\right]
_{\overline{z}}(\overline{g}^2-1)}{R(g)^2}+2\frac{\overline{g}
|g_z|^2}{R(g)}.
\]
Since $g$ satisfies \eqref{gauss}, the above equation yields
\begin{equation}
\label{eq:corol3.16A}
(A_1)_{\overline{z}}=\frac{|g_{z}|^2}{|R(g)|^2}\overline{\left(
2gR(g)-R_q(g)(g^2-1)\right)}.
\end{equation}
Now, equation \eqref{eq:potentialunim} implies that
\begin{equation}
\label{eq:corol3.16B} 2gR(g)-R_q(g)(g^2-1)=-2i|g|^2\left[ \left(
g+\frac{1}{g}\right) (\mu _3+iH)+ \left( \overline{g}+
\frac{1}{\overline{g}}\right) (\mu _2+iH)\right] ,
\end{equation}
where $\mu _1,\mu _2,\mu _3$ are given by (\ref{def:mu}) in terms of
the structure constants $c_1,c_2,c_3$ of the unimodular metric Lie
group $X$.

Substituting \eqref{eq:corol3.16B} into \eqref{eq:corol3.16A} and
using again \eqref{eq:Ai} we arrive to
\begin{equation}
\label{eq:corol3.16C} (A_1)_{\overline{z}}=A_2\overline{A_3}(\mu
_3-iH)-A_3\overline{A_2}(\mu _2-iH).
\end{equation}
Working analogously with $A_2,A_3$, we obtain
\begin{equation}
\label{eq:corol3.16D} \left\{
\begin{array}{l}
(A_1)_{\overline{z}}=A_2\overline{A_3}(\mu
_3-iH)-A_3\overline{A_2}(\mu _2-iH), \\
(A_2)_{\overline{z}}=A_3\overline{A_1}(\mu
_1-iH)-A_1\overline{A_3}(\mu _3-iH),
\\
\rule{0cm}{.4cm} (A_3)_{\overline{z}}=A_1\overline{A_2}(\mu
_2-iH)-A_2\overline{A_1}(\mu _1-iH).
\end{array}
\right.
\end{equation}
In particular, from (\ref{def:mu}) we have
\begin{equation}
\label{eq:corol3.16E} \left\{
\begin{array}{l}
(A_1)_z-(\overline{A_1})_z=c_1(A_2\overline{A_3}-
A_3\overline{A_2}),
\\
\rule{0cm}{.4cm} (A_2)_z-(\overline{A_2})_z=c_2(A_3\overline{A_1}-
A_1\overline{A_3}),
\\
\rule{0cm}{.4cm} (A_3)_z-(\overline{A_3})_z=c_3(A_1\overline{A_2}-
A_2\overline{A_1}).
\end{array}
\right.
\end{equation}
Therefore, if we define $\cA:=\sum_{i=1}^3 A_i e_i$ where
$e_i:=(E_i)_e$ and $e$ is the identity element in $X$, we have by
\eqref{eq:corol3.16E} that
\begin{equation}
 \label{eq:new1}
\cA_{\overline{z}} - (\overline{\cA})_z =[\cA,\overline{\cA}].
\end{equation}

We want to solve the first order PDE ${\mathcal A}:=f^{-1}f_z$ in
the unknown $f$. To do this, we will apply the classical Frobenius
Theorem. By a direct computation, we see that the formal
integrability condition $(f_z)_{\bar{z}}=(f_{\bar{z}})_z$ for the
local existence of such map $f$ is given precisely by
\eqref{eq:new1}. Thus, as $\Sigma$ is simply connected, the
Frobenius theorem ensures that there exists a smooth map
$f\colon\Sigma\flecha X$ such that $\cA=f^{-1} f_z$ and  $f$ is
unique once we prescribe an initial condition $f(z_0)=p_0\in X$.

A straightforward computation  gives that $\esiz df,df\esde =\l
|dz|^2$, where the conformal factor $\l $ is
 \begin{equation}\label{formulanda}
 \l= \frac{4(1+|g|^2)^2}{|R(g)|^2} \, |g_z|^2
\colon \Sigma\flecha [0,\8).
\end{equation}
Since $g_z\neq 0$ and $R$ does not vanish on $g(\Sigma)$, it is easy
to see that $\landa>0$ on $\Sigma$ by using that $R(q)/|q|^4$ has a
finite limit at $q=\8$ (which is non-zero, by hypothesis). Thus, $f$
is a conformally immersed surface in $X$, and it is immediate from
\eqref{eq:Ai} that the stereographically projected left invariant
Gauss map of $f$ is $g$. Also, it is clear that $f$ is unique up to
left translations in $X$.

It only remains to prove that $f$ has constant mean curvature of
value $H$. Let $\cH$ be the mean curvature function of $f$. By the
Gauss-Weingarten formulas, we have
\begin{equation}
\label{eq:3.3} \nabla _{f_{\overline{z}}}f_z=\frac{\l \cH}{2}N,
\end{equation}
where $\nabla$ is the Levi-Civita connection in $X$ and $N$ is the
unit normal field to $f$ in $X$. Thus, since $f_z=\sum_{i=1}^3 A_i
E_i$, if we let $N=\sum_{i=1}^3 N_i E_i$ then we have
\begin{equation}
\label{eq:3.6} \sum_{i=1}^3 (A_i)_{\bar{z}} E_i + \sum_{i,j=1}^3
\overline{A_i}A_j\nabla _{E_i}E_j =
 \frac{\l
\cH}{2}\sum_{i=1}^3 N_i E_i.
\end{equation}
Therefore, using (\ref{eq:LCunim}) and that $N=\frac{-2i}{\landa}
(\sum_{i=1}^3 A_i E_i)\times (\sum_{i=1}^3 \overline{A_i} E_i)$
where $\times$ is the cross product in the Lie algebra of $X$, we
conclude that
\begin{equation}
\label{corol3.16D} \left\{
\begin{array}{l}
(A_1)_{\overline{z}}=A_2\overline{A_3}(\mu
_3-i\cH)-A_3\overline{A_2}(\mu _2-i\cH), \\
(A_2)_{\overline{z}}=A_3\overline{A_1}(\mu _1-i
\cH)-A_1\overline{A_3}(\mu _3-i\cH),
\\
\rule{0cm}{.4cm} (A_3)_{\overline{z}}=A_1\overline{A_2}(\mu
_2-i\cH)-A_2\overline{A_1}(\mu _1-i\cH).
\end{array}
\right.
\end{equation}
Comparing this equation with \eqref{eq:corol3.16D}, we get $H=\cH$,
as desired.

Conversely, we need to show that the left invariant Gauss map $g$ of
any conformally immersed $H$-surface $f\colon \Sigma \looparrowright
X$, where $H$ is such that the $H$-potential of $X$ does not vanish
on $g(\Sigma )$, satisfies the elliptic PDE \eqref{gauss} and also
$g_z\neq 0$ everywhere. But this can be done just as in the proof of
Theorem~3.4 in the paper by Daniel and Mira~\cite{dm2}, using now
the general metric relations given by \eqref{eq:LCunim} instead of
the specific ones of $\sol$. We omit the details.
\end{proof}

The following application of Theorem~\ref{rep} will be useful later
on.
\begin{corollary}
\label{invacor} Let $f\colon\Sigma\looparrowright X$ be an immersed
$H$-surface with left invariant Gauss map $g\colon \Sigma \to
\overline{\C }$  and assume
that the $H$-potential for $X$ does not vanish on $g(\Sigma )$. 
Then, the differential of 
$g$ 
has rank at most $1$ everywhere on $\Sigma $ if and only if $f$ is
invariant under the flow of a right invariant vector field on $X$.
Furthermore, if $f$ is invariant under the flow of a right invariant
vector field, then the rank of the differential of $g$ is $1$
everywhere on $\Sigma $ and the Gauss map image $g(\Sigma)$ is a
regular curve in $\cb$.
\end{corollary}
\begin{proof}
First assume that $g$ has rank less than or equal to $1$ everywhere
on $\Sigma $. Since we are assuming that the $H$-potential of $X$
does not vanish on $g(\Sigma )$, then the last statement in
Theorem~\ref{rep} implies $g_z\neq 0$. Therefore, the differential
of $g$ has rank one everywhere and $g(\Sigma)$ is a regular curve.
Around any given point $z_0\in \Sigma$, there exists a conformal
parameter $z=s+it$ such that $g=g(s)$ with $g'(s)\neq 0$, i.e.,
$g(s+it)$ does not depend on $t$. In particular, $g$ can be
considered to be defined on a vertical strip in the $(s,t)$-plane,
and by Theorem~\ref{rep} the conformally immersed $H$-surface
$f(s,t)$ with Gauss map $g$ can also be extended to be defined on
that vertical strip in the $(s,t)$-plane, so that $f(s,t+t_0)$
differs by a left translation from $f(s,t)$ for every $t_0\in \R$.
Hence, by analyticity, we see that the $\Sigma$ is invariant under
the flow of a right invariant vector field of $X$.

The converse implication is trivial, since any surface invariant
under the flow of a right invariant vector field is locally obtained
by left translating in $X$ some regular curve on the surface, and so
the differential of its left invariant Gauss map has rank at most 1
everywhere.
\end{proof}

\section{Index-one spheres in  metric Lie groups.}
\label{sec:index1}

The \emph{Jacobi operator} of an immersed
 two-sided hypersurface $M$ with constant
 mean curvature in a
Riemannian manifold $\overline{M}^{n+1}$ is defined as
\begin{equation}
\label{eq:Jacobi}
\mathcal{L}=\Delta +|\sigma |^2+\mbox{Ric}(N).
\end{equation}
Here, $\Delta $ is the Laplacian of $M$ in the induced metric,
$|\sigma |^2$ the square of the norm of the second fundamental form
of $M$ and $N$ a unit vector field on the hypersurface.

A domain $\Omega \subset M$ with compact closure is said to be {\it
stable} if $-\int _Mu\mathcal{L}u\geq 0$ for all compactly supported
smooth functions $u\in C^{\infty }_0(\Omega )$, or equivalently, for
all functions in the closure $H_0^1(\Omega )$ of $C^{\infty
}_0(\Omega )$ in the usual Sobolev norm. $\Omega $ is called {\it
strictly unstable} if it is not stable.

The {\it index} of a domain $\Omega \subset M$ with compact closure
is the number of negative eigenvalues of $-\mathcal{L}$ acting on
functions of $H_0^1(\Omega )$; thus, $\Omega $ is stable if and only
if its index is zero. If zero is an eigenvalue of $-\mathcal{L}$ on
$H_0^1(\Omega )$, then the {\it nullity} of $\Omega $ is the
(finite) dimension of the eigenspace associated to this zero
eigenvalue. Since the index of stability is non-decreasing with the
respect to the inclusion of subdomains of $M$ with compact closure,
one can define the {\it index of stability} of $M$ as the supremum
of the indices over any increasing sequence of subdomains $\Omega
_i\subset M$ with compact closure and $\cup _i\Omega _i=M$.

If $M$ is compact, then $M$ is called {\it weakly stable} if $-\int
_Mu\mathcal{L}u\geq 0$ for all $u\in C^{\infty }(M)$ with $\int
_Mu=0$. Every solution to the classical isoperimetric problem is
weakly stable, and every compact, weakly stable constant mean
curvature surface has index zero or one for the stability operator.
We refer the reader to the handbook~\cite{mpr19} for the basic
concepts and results concerning stability  properties of constant
mean curvature hypersurfaces in terms of the Jacobi operator.

Since every metric Lie group $X$ is orientable, then the
two-sidedness of a surface in $X$ is equivalent to its
orientability. It is easy to see that the index of an $H$-sphere
$S_H$ in $X$ is at least one; indeed, if $\{F_1,F_2,F_3\}$ denotes a
basis of right invariant vector fields of $X$ (that are Killing,
independently of the left invariant metric on $X$), then the
functions $u_i=\esiz F_i,N\esde$, $i=1,2,3$, are \emph{Jacobi
functions} on $S_H$, i.e., $\mathcal{L}u_i=0$. Since right invariant
vector fields on $X$ are identically zero or never zero and spheres
do not admit a nowhere zero tangent vector field, the functions
$u_i=\esiz F_i,N\esde$, $i=1,2,3$, are linearly independent. Hence,
$0$ is an eigenvalue of $-\cL$ of multiplicity at least three. As
the first eigenvalue is simple, then $0$ is not the first eigenvalue
of $-\cL$ and thus, the index of $S_H$ is at least one. Moreover, if
the index of $S_H$ is exactly one, then it follows from Theorem~3.4
in Cheng~\cite{cheng1} (see also~\cite{dm2,ross2})
 that ${\rm Ker} (\cL)$ has dimension three, and hence we conclude that
\begin{equation} \label{nul} {\rm
Ker} (\cL)={\rm Span }\{u_1,u_2,u_3\}. \end{equation}

Our next goal is to describe some structure results for  the space
of index-one spheres in a metric Lie group.

\begin{theorem}
\label{thm:index1} Any index-one $H$-sphere $S_H$ in a
three-dimensional, simply-connected metric Lie group $X$ satisfies:
\ben[(1)]
\item The left invariant Gauss map of $S_H$ is an orientation preserving diffeomorphism
to $\S^2$.

\item   $S_H$ is unique up to left translations among
$H$-spheres in $X$.

\item $S_H$ lies inside a real-analytic family $\{S_{H'} \mid H'\in
(H-\varepsilon , H+\varepsilon )\}$ of index-one spheres in $X$ for
some $\ve>0$, where $S_{H'}$ has constant mean curvature of value
$H'$. Furthermore, for $H$ sufficiently large, the range of values
of the mean curvature for the maximal real analytic family $\{
S_{H'}\} _{H'}$ containing $S_H$ contains the interval $(H-\ve
,\infty )$. In particular, there exists a unique component
$\mathcal{C}$ of the space of index-one spheres with constant mean
curvature in $X$ such that the values of the mean curvatures of the
spheres in $\mathcal{C}$ are not bounded from above.

\item For any $H_0\geq 0$, there exists a constant $A(H_0)>0$ such that the
norm of the second fundamental form of every index-one $H'$-sphere
in $X$ with $0\leq |H'|\leq H_0$ is bounded from above by $A(H_0)$.

\item If the map $\cH\colon\mathcal{C}\to \R$ that assigns to each sphere in $\mathcal{C}$ its mean curvature
is not surjective, there exists a certain $h_0(X)\geq 0$ such that
for every $H > h_0(X)$ there is an index-one sphere $S_H\in
\mathcal{C}$, and the areas of any sequence $\{ S_{H_n}\} _n\subset
\mathcal{C}$ with $H_n\searrow h_0(X)$ satisfy
 $\lim_{n\to \infty } {\rm Area} (S_{H_n}) =\8$.
In particular, if $\cH$ is not surjective, then
$\cH(\cC)=(h_0(X),\infty)\subset (0,\infty)$.
\item If $X$ is not isomorphic to
$\su$, then it does not admit any compact immersed minimal surfaces,
and thus, $\cH$ is not surjective in this case.
 \een
\end{theorem}

\begin{proof}
We will divide the proof into  steps. Let $S_H$ denote an index-one
$H$-sphere in $X$.

\vspace{0.2cm}

{\bf Step 1:} \emph{The left invariant Gauss map of $S_H$ is an
orientation-preserving diffeomorphism onto $\overline{\C}$ (hence
item~(1) of the theorem holds).}
\begin{proof}[Proof of Step~1]
Let $\cG\colon S_H\equiv \overline{\C} \to \overline{\C}$ be the
left invariant Gauss map of $S_H$, which we view as a conformal
immersion $f\colon\overline{\C}\looparrowright X$. Observe that the
$H$-potential $R$ for $X$ does not vanish, by Fact \ref{fact3}.

We first prove that $\cG$ is a diffeomorphism. By elementary
covering space theory, it suffices to check that $\cG$ is a local
diffeomorphism. Arguing by contradiction, assume this condition
fails at a point $z_0\in S_H$, and observe that $d\cG(z_0)\neq 0$
since $\cG$ is nowhere antiholomorphic  by Theorem~\ref{rep} (also
recall that the $H$-potential for $X$ satisfies $R(\cG (z_0))\neq 0$
by Fact~\ref{fact3}).
Therefore, there exists a conformal coordinate $z=x+iy$ near $z_0$
so that $\cG_x (z_0)=0$ and $\cG_y (z_0)\neq 0$. We may assume up to
conformal reparameterization of $S_H$ that $z_0=0$.

Consider the second order ODE given by specializing  \eqref{gauss}
to functions of the real variable $y$, i.e.,
\begin{equation}
\label{eq:thm3.14*}
\widehat{g}_{yy}=\frac{R_q}{R}(\widehat{g})(\widehat{g}_y)^2+\left(
\frac{R_{\overline{q}}}{R}-\frac{\overline{R_q}}{\overline{R}}\right)
(\widehat{g})|\widehat{g}_y|^2.
\end{equation}

Let $\widehat{g}=\widehat{g}(y)$ be the (unique) solution of
(\ref{eq:thm3.14*}) with initial conditions $\widehat{g}(0)=\cG(0)$,
$\widehat{g}_y(0)=\cG_y(0)$. From Theorem~\ref{rep}, there exists a
conformally immersed $H$-surface $\widehat{f}$ in $X$ with left
invariant Gauss map $\widehat{g}$, and such that
$f(0)=\widehat{f}(0)$. By Corollary~\ref{invacor}, $\widehat{f}$ is
invariant under the flow of a right invariant Killing field $F$ on
$X$. In particular, the function $\widehat{u}=\langle
\widehat{N},F\rangle $ vanishes identically, where $\widehat{N}$ is
the unit normal vector field to $\widehat{f}$.

Now observe that the left invariant Gauss maps $\cG, \widehat{g}$ of
$f,\widehat{f}$ satisfy $\cG(0)=\widehat{g}(0)$ and $\cG_z (0)=
\widehat{g}_z (0)$, $\cG_{\overline{z}} (0)=
\widehat{g}_{\overline{z}} (0)$. In addition, by Theorem~\ref{rep}
one can see that the first and second fundamental forms of a
conformally immersed $H$-surface in $X$ at any given point are
totally determined by the values of $g, g_z$ and $g_{\overline{z}}$
at that point, where $g$ is the left invariant Gauss map of the
surface. Hence, we see that $f,\widehat{f}$ have at least a second
order contact at $z=0$. Denoting $u=\esiz N,F\esde
\colon\overline{\C}\flecha \R$, we conclude that
$u(0)=\widehat{u}(0)=0$ and $u_z (0)=\widehat{u}_z (0)=0$, but $u$
is not identically zero on $\overline{\C}\equiv S_H$ (since there
are no nowhere vanishing tangent vector fields on a two-dimensional
sphere).

Now observe that since $u$ is a solution of the elliptic PDE $\cL u
=0$, a classical argument by Bers (see for instance Theorem~2.5 in
Cheng~\cite{cheng1}) proves that the condition $u(0)=u_z(0)=0$
implies that the nodal set $u^{-1} (0)$ of $u$ around $0$ consists
of $n\geq 2$ analytic arcs that intersect transversely at $0$. This
is a contradiction with Corollary 3.5 in~\cite{cheng1}, which states
in our situation that if the sphere $S_H$ has index-one, then
$u^{-1}(0)$ is an analytic simple closed curve in $S_H$. This
contradiction completes the proof that the left invariant Gauss map
of $S_H$ is a diffeomorphism.

As we have already observed, Theorem~\ref{rep} implies that the
mapping $\cG$ is never anti-holomorphic. But a simple consequence of
the Hopf index theorem applied to  line  fields on the sphere is
that every  diffeomorphism $F$ of the Riemann sphere to itself has
some point where it is either conformal (when $F$ is orientation
preserving) or anticonformal (when $F$ is orientation reversing).
Hence, the diffeomorphism $\cG$ is orientation preserving, which
completes the proof of Step~1.
\end{proof}

{\bf Step 2:} \emph{Any $H$-sphere in $X$ is a left translation of
the index-one $H$-sphere $S_H$. In particular, $S_H$ is unique up to
ambient isometries in $X$ and thus, item~(2) of the statement of
Theorem~\ref{thm:index1} holds.}

\begin{proof}[Proof of Step~2]
Using the result from Step~1 that the left invariant Gauss map $\cG
$ of $S_H$ is a diffeomorphism, this step can be carried out just as
in Daniel-Mira~\cite{dm2}, and so we will only sketch the idea.
Define $L\colon\overline{\C}\flecha \C$ implicitly in terms of the
diffeomorphism $\cG $ and of the $H$-potential $R$ of $X$ (which
does not vanish, since the sphere $S_H$ exists, see
Fact~\ref{fact3}) as
\[
L(\cG(z))= -\frac{\overline{\cG}_z}{R(\cG(z)) \cG_z }.
\]
Given a conformally immersed $H$-surface
$f\colon\Sigma\looparrowright X$ with left invariant Gauss map $g$,
define the complex quadratic differential
\[
Q_H (dz)^2= \left( L(g) g_z^2  + \frac{1}{R(g)} g_z \overline{g}_z
\right) (dz)^2.
\]
By the left invariant Gauss map equation~(\ref{gauss}) and the fact
that $\cG$ is a diffeomorphism, it can be proved that:
 \begin{enumerate}[(1)]
   \item
If $Q_H$ does not vanish identically for an $H$-surface $f$, then
$|(Q_H)_{\overline{z}}|/ |Q_H|$ is locally bounded on the surface.
Consequently, $Q_H$ has only isolated zeros of negative index (see
e.g., Alencar, do Carmo and Tribuzy~\cite{act1} or Jost~\cite{Jo}).
 \item
 $Q_H=0$ for an $H$-surface $f$ if and only if $f$ is, up to
left translation, an open piece of the index-one sphere $S_H$.
 \end{enumerate}
As the first case above is impossible on a topological sphere by the
Poincaré-Hopf index theorem, we conclude  that any $H$-sphere in $X$
is a left translation of $S_H$.
\end{proof}

\begin{remark}
\label{fixed-point}
 {\em
Let $S_H$ be an index-one $H$-surface immersed in $X$. The
uniqueness conclusion in Step~2 above implies that  $S_H$  is
symmetric in the following sense. Suppose $\phi\colon X \to X$ is a
non-trivial orientation preserving isometry of  $X$ with at least
one fixed point $p$. After conjugating $\phi $ with the left
translation by $p$, we can assume $p=e$. Since $\phi $ is
orientation preserving in $X$ there exists a unitary vector $v\in
T_eX$ such that $d\phi _e(v)=v$. Left translate $S_H$ so that $e\in
S_H$ and the unit normal of $S_H$ at $e$ agrees up to sign with $v$.
Then, both $S_H$ and $\phi(S_H)$ have the same normal at $p$ and
Step~2 implies that $\phi(S_H)=S_H$, hence $\phi$ induces an
orientation preserving isometry $\wh{\phi }=\phi |_{S_H}$ of $S_H$.
In other words, every orientation  preserving element of the
isotropy group at $e$ of the isometry group of $X$ induces an
orientation preserving isometry of some left translation of $S_H$.}

\end{remark}

\vspace{.2cm} {\bf Step 3:} \emph{Any index-one $H$-sphere $S_H$ in
$X$ lies inside a real-analytic family $\{S_{H'} \mid H'\in
(H-\varepsilon , H+\varepsilon )\}$ of index-one spheres in $X$,
where $S_{H'}$ has constant mean curvature  $H'$. }

\begin{proof}[Proof of Step~3] The proof is an adaptation of the proof of Proposition~5.6
in~\cite{dm2}. Recall from (\ref{nul}) that ${\rm Ker} (\cL) = {\rm
Span} \{u_1,u_2,u_3\}$ for any index-one sphere $S_H$ in $X$. In
this setting, a standard application of the implicit function
theorem to the following modified mean curvature functional
$\mathcal{K}$ proves that $S_H$ can be analytically deformed through
constant mean curvature spheres in $\mathcal{K}^{-1} (\{0\})$ by
moving the mean curvature parameter:
\[
\begin{array}{rcl}
\mathcal{K}\colon
C^{2,\a }(S_H)_{\ve }\times \R^3&\longrightarrow & C^{0,\a }(S_H)\\
(\varphi ,a_1,a_2,a_3)& \longmapsto  & H(\Sigma _{\varphi })-\sum
_{i=1}^3a_i\langle N_{\varphi },F_i\rangle ,
\end{array}
\]
where $C^{2,\a }(S_H)_{\ve }$ denotes a sufficiently small
neighborhood of $0\in C^{2,\a }(S_H)$, $\a \in (0,1)$, so that for
every $\varphi \in C^{2,\a }(S_H)_{\ve }$, the normal graph of
$\varphi $ over $S_H$ defines an immersed $C^{2,\a }$ sphere $\Sigma
_{\varphi }$ in $X$, $H(\Sigma _{\varphi })$ is the mean curvature
of $\Sigma _{\varphi }$ (the unit normal vector field $N_{\varphi }$
of $\Sigma _{\varphi }$ is chosen so that it coincides with the
original one on $S_H$ when $\varphi =0$), and $F_1,F_2,F_3$ is the
basis of right invariant vector fields $F_1,F_2,F_3$ whose normal
components are the $u_i$.

All the deformed constant mean curvature spheres $S_{H'}$ in
$\mathcal{K}^{-1}(\{ 0\} )$ have index one since otherwise, an
intermediate value argument would lead to an $H'$-value for which
$S_{H'}$ has nullity at least four, which is impossible
by~(\ref{nul}).
\end{proof}

 {\bf Step 4:} \emph{ The proofs of items (3) and (4)
of Theorem~\ref{thm:index1}.}
\par \vspace{.2cm} {\sc Proof of Step 4.}
It is well-known that for $t>0$ sufficiently small, solutions to the
isoperimetric problem in $X$ for volume $t$  exist and geometrically
are small, almost-round balls with boundary spheres $S(t)$ of
constant mean curvature approximately $\left( \frac{4\pi
}{3t}\right) ^{1/3}$. All such $H$-spheres $S(t)$ have index one,
since the existence of two negative eigenvalues of $-\cL$  would
contradict their weak stability. Hence, there is some
$\widetilde{h}_0(X)>0$ such that for every $H\in
(\widetilde{h}_0(X),\infty)$, there is an index-one $H$-sphere. In
particular, by Step~3, there exists a maximal real-analytic family
$\mathcal{C}$ of index-one spheres in $X$ where the set of values of
the mean curvature of its elements contains the interval
$(\widetilde{h}_0(X), \infty )$. Note by Step~2, up to left
translations, $\mathcal{C}$ contains a unique $H$-sphere for each
image value of the mean curvature function $\cH \colon
\mathcal{C}\to [0,\infty )$. This finishes the proof of item~(3) of
the theorem.

We next prove item~(4). Arguing by contradiction, assume that there
exists $H_0\geq 0$ and a sequence of index-one $H_n$-spheres
$S_{H_n}$ with $H_n\in [0,H_0]$, such that the sequence of second
fundamental forms $\{ \sigma_{S_{H_n}}\} _n$ is not uniformly
bounded. Then, after left translating $S_{H_n}$ by the inverse
element in $X$ of the point $p_n$ of $S_{H_n}$ where the norm of its
second fundamental form is maximal, we may assume that $p_n=e\in
S_{H_n}$ is the point where $|\sigma_{S_{H_n}}|$ has a maximal value
$\l_n$. After passing to a subsequence, we may assume that $\l_n\to
\8$.

Fix  $\ve \in (0,1)$ smaller than the injectivity radius of $X$ and
consider exponential coordinates in the extrinsic ball $B_X(e,\ve )$
centered at $e$ of radius $\ve $, associated to the orthonormal
basis $(E_1)_e,(E_2)_e,(E_3)_e$ of $T_eX$. For $n$ large, rescale
the metric $\langle ,\rangle $ of $X$ to $\l _n^2\langle ,\rangle $,
obtaining a new metric Lie group $X_n=(X,\l _n^2\langle ,\rangle )$.
Observe that $(E_i^n)_e=\l _n^{-1}(E_i)_e$, $i=1,2,3$, defines an
orthonormal left invariant basis for $X_n$, and that the exponential
coordinates in $X_n$ associated to this orthonormal basis are
nothing but rescaling the original coordinates in $X$ by $\l _n$. As
$\l_n \to \infty $, the related left invariant metrics on $X_n$ are
converging uniformly and smoothly to a flat metric in the limit
$\rth$-coordinates. Also, the surface $S_{H_n}$ has constant mean
curvature $\l _n^{-1}H_n$ in $X_n$ and its second fundamental form
is bounded by $1$. In particular, there exists a fixed size
neighborhood of $e=\vec{0}$ at $T_eS_{H_n}$ (i.e., the size is
independent of $n$) such that $S_{H_n}$ can be expressed around
$e\in S_{H_n}$ as the graph of a function defined in this
neighborhood. Elliptic theory then gives that (a subsequence of) the
$S_{H_n}$ converge to a minimal graph $M_{\infty } \subset \R^3$
passing through $\vec{0}$ which has Gaussian curvature $-1$ at this
point, where by \emph{graph} we mean that $M_{\infty}$ is a graph
over its projection to its tangent plane at $\vec{0}$. In
particular, since $M_{\infty }$ is minimal and its Gaussian
curvature at the origin is $-1$, then the (left invariant) Gauss map
of $M_{\infty }$ considered to be a surface in  the metric Lie group
$\rth$ (i.e. the classical Gauss map for oriented surfaces in
$\R^3$) is an {\em orientation reversing} diffeomorphism in a small
neighborhood of the origin.

As $n\to \infty$, the norms in $X_n$ of the Lie brackets of any two
of the elements in the orthonormal basis of left invariant vector
fields corresponding to the vectors $\{
(E_1^n)_e,(E_2^n)_e,(E_3^n)_e\}$ converge to~0, and so the metric
Lie group structures of the $X_n$ converge to the flat left
invariant metric of the abelian Lie group $\R^3$ and $\lim_{n\to
\infty } (E_i^n)_e=\frac{\partial}{\partial  x_i}(\vec{0})$ for
$i=1,2,3$, where $(x_1,x_2,x_3)$ are the limit $\R^3$-coordinates.
Then by construction of the left invariant Gauss map, the
differentials at $e$ of the left invariant Gauss maps of $S_{H_n}$
in $X_n$ converge to the differential of the classical Gauss map of
$M_{\infty }$ at $\vec{0}$. This fact contradicts that, by item~(1)
in the theorem, the Gauss maps of the $S_{H_n}$ are orientation
preserving for all $n$.
\end{proof}

{\bf Step 5:} \emph{ The proof of item~(5) of
Theorem~\ref{thm:index1}.}
\begin{proof}[Proof of Step 5]
Suppose that $\cH \colon \mathcal{C}\to \R$ is not onto. Let
$h_0(X)$ be the infimum of $\cH $ on $\mathcal{C}$. Note that by
Step~3, $h_0(X)$ is not achieved in $\mathcal{C}$. Take any sequence
$\{ S_{H_n}\} _n\subset \mathcal{C}$ with mean curvatures
$H_n\searrow h_0(X)$. If the areas of the spheres $S_{H_n}$ are
uniformly bounded, then because the norms of the second fundamental
forms of this sequence of spheres are uniformly bounded by the
already proven item~(4) of Theorem~\ref{thm:index1}, a standard
compactness argument in elliptic theory implies that there exists an
immersed $h_0(X)$-sphere $S$ in $X$, which is a limit of left
translations of the $S_{H_n}$. By straightforward modification of
the arguments in Step 3, $S$ has index one and belongs to
$\mathcal{C}$, which contradicts the definition of $h_0(X)$.
\end{proof}

{\bf Step 6:} \emph{ The proof of item~(6) of
Theorem~\ref{thm:index1}.}
\begin{proof}[Proof of Step 6]
If $X$ is not isomorphic to $\su$, then $X$ contains no compact
minimal surfaces; this follows from the proof of Fact \ref{fact3}. This completes
the proof of Theorem~\ref{thm:index1}.
\end{proof}

We will prove next the Alexandrov embeddedness of any index-one
$H$-sphere in a three-dimensional metric Lie group $X$ which belongs
to the component $\mathcal{C}$ defined in item~(3) of
Theorem~\ref{thm:index1}. For this, we prove first the following
general result:

\begin{lemma}
\label{alex-lemma}
Let $h_t\colon\S^2 \la X$, $t\in [a,b]$, be a smooth family of
immersions of the sphere $\S^2$ such that \ben[(1)]
\item $h_a$ extends to an isometric
immersion of a compact ball $\wt{h}_a\colon E_a\to X$, where
$\partial E_a= \S^2$.
\item The boundary of $E_a$ has positive constant mean curvature with
respect to the unit normal vector pointing towards $E_a$. In the
sequel, we orient $\esf ^2$ with respect to the inward pointing unit
normal vector to $E_a$.
\item The immersed spheres $h_t(\S^2)$ have positive constant mean curvature
(smoothly varying with $t\in[a,b]$). \een Then there exists a smooth
1-parameter family of isometric immersions $\{ \wt{h}_t\colon E_t\to
X\ | \ t\in [a,b]\} $ of Riemannian balls $E_t$ that are mean convex
and extend the immersions $h_t=\wt{h}_t|_{\partial E_t}$, where
$\S^2$ is the oriented boundary of $E_t$.
\end{lemma}
\begin{proof}
Consider the set of values
\[
\cT=\{ t_0\in [a,b]\ : \mbox{ the family $\{\wt{h}_t\}_t$ in the
conclusion of the lemma exists for all $t\in [a,t_0]$} \} .
\]
Our goal is to show that $\cT =[a,b]$ which will prove the lemma.
Note that the set $\cT$ is non-empty since $a\in \cT$ by hypothesis,
and that $\cT$ is an open subset of $[a,b]$, since its defining
properties are  open ones for immersions of compact three-manifolds
with boundary and by condition~(3), the oriented extended immersions
have positive mean curvature spheres as their oriented boundaries.
By the connectedness of $[a,b]$, it suffices to prove that the
connected component of $\cT$ containing $a$ is a closed interval.
Arguing by contradiction, suppose that this component has the form
$[a,c)$ for some $c\in (a,b]$.

Consider a sequence $\{ t_n\} _n\subset \cT \cap [a,c)$ with $t_n\to
c$ and related immersions $\wt{h}_{t_n}\colon E_{t_n} \to X$.  Since
$\{ h_t\ | \ t\in [a,b]\} $ is a smooth compact family of
immersions, then the second fundamental forms of the spheres
$h_{t_n}(\S^2)$ are uniformly bounded. As $c\not\in \cT$, then the
surfaces $\partial E_{t_n}$ must fail to have  uniform size 1-sided
regular neighborhoods in their respective balls $E_{t_n}$ as $n\to
\infty$. But this is a contradiction with the One-Sided Regular
Neighborhood Theorem in~\cite{mt3}. This contradiction proves that
$\mathcal{T}=[a,b]$ and finishes the proof of the lemma.
\end{proof}

\begin{corollary}
\label{alex-cor}
Let $X$ be a three-dimensional metric Lie group, and $\mathcal{C}$
be the connected component of the space of oriented, index-one,
constant mean curvature spheres defined in item~(3) of
Theorem~\ref{thm:index1}. Then, any oriented $H$-sphere $S_H\in
\mathcal{C}$ is Alexandrov embedded in the sense of
item~(\ref{alex-embed}) of Theorem~\ref{main}.

\end{corollary}
\begin{proof}
Let $S_H\in \mathcal{C}$ be a sphere of positive constant mean
curvature $H$. By the proof of Theorem~\ref{thm:index1}, there
exists a real analytic one-parameter family $\{S_{H'} : H'\in
[H,\8)\}\subset \mathcal{C}$ of $H'$-spheres which starts at $S_H$
when $H'=H$. As explained in Step 4 of the proof of
Theorem~\ref{thm:index1}, for $H'$ large enough the spheres $S_{H'}$
are embedded, and therefore they  are the oriented boundaries of
mean convex balls in $X$. By Lemma~\ref{alex-lemma} we deduce that
$S_H$ is Alexandrov embedded in the sense of Theorem~\ref{main}.
This proves the corollary for $H> 0$. If $H<0$ for a $H$-sphere
$S_H\in \mathcal{C}$, then
 after a change of orientation we get the same
conclusion. Finally, assume that $S_0\in \mathcal{C}$ is a minimal
sphere. By item~(6) of Theorem \ref{thm:index1}, $X$ is isomorphic to
$\su$. It then follows that $S_0$ is embedded, since by
\cite{smith1} there exists an embedded minimal sphere $\hat{S}$ in
$X$, and by item~(2) of Theorem \ref{thm:index1} $S_0$ and $\hat{S}$
differ by an ambient isometry (note that $S_0$ has index one since
$S_0\in \mathcal{C}$). This finishes the proof of the corollary.
\end{proof}

\begin{remark}
\label{rem44} {\em If there exists a minimal sphere $S_0$ in
$\mathcal{C}$ (which can only happen if $X$ is isomorphic to $\su$
by item~(6) of Theorem~\ref{thm:index1}), then by the above proof,
$S_0$ is the oriented boundary of a mean convex ball in $X$. Hence,
after changing the orientation, $S_0$ is the boundary of a second
such mean convex ball on its other side.}
\end{remark}

\section{Invariant
limit surfaces of index-one spheres.} \label{sec:limits}

\begin{definition}
\label{def6.1} {\rm
 We will denote by  $\chM _X$ the moduli space of all index-one
$H$-spheres in a metric Lie group $X$ with $H\in \R $ varying.
 We say that a complete, non-compact, connected
$H$-surface $f\colon\Sigma\looparrowright X$ is a \emph{limit
surface} of $\chM _X$ with base point $p\in \Sigma $ (also called a
{\em pointed limit immersion} and denoted by $f\colon (\Sigma
,p)\looparrowright (X,e)$) if $e=f(p)$ and there exists a sequence
$\{ F_n\colon S_n\la X\} _n\subset \chM _X$, compact domains
$\Omega_n\subset S_n$ and points $p_n\in \Omega _n$ such that the
following two conditions hold:
 \begin{enumerate}[(1)]
   \item
 $f$ is a limit of the immersions $f_n=l_{F_n(p_n)^{-1}}\circ (F_n|_{\Omega _n})\colon \Omega_n
 \la X$ obtained by left translating $F_n|_{\Omega _n}$ by the inverse of $F_n(p_n)$ in $X$ (hence $f_n(p_n)=e$).
 Here, the convergence is the uniform convergence in the
 $C^k$-topology for every $k\geq 1$, when we view the surfaces as
 local graphs in the normal bundle of the limit immersion.
  \item
 The area of $f_n$ is greater than $n$.
 \end{enumerate}}
\end{definition}

Recall from Section~\ref{sec:index1} that in any metric Lie group
$X$, exactly one of the following two possibilities occurs:
\begin{enumerate}[(a)]
\item There exists a uniform upper bound of the areas of all
   index-one, constant mean curvature spheres
   in $X$. In this case, limit surfaces of $\chM _X$ do not exist since
   property (2) in Definition~\ref{def6.1} cannot occur.

\item There exists $h_0(X)\geq 0$ such that $\chM _X$ contains
   $H$-spheres for all  $H\in (h_0(X),\8)$, and the
   areas of any sequence of $H_n$-spheres $\{ S_n\} _n\subset \chM _X$ with
   $H_n\searrow h_0(X)$ are unbounded. In this case, the mean
   curvature of any limit surface of the sequence
   $\{ S_n\} _n$ in $\chM _X$ is
   $h_0(X)$.
\end{enumerate}

\begin{lemma}
\label{DeltaF} Suppose that $X$ is a metric Lie group for which
case~{\rm (b)} above holds. Then, there exists a pointed limit
immersion $F\colon (\Sigma ,p)\looparrowright (X,e)$ of $\chM_X$
with constant mean curvature $h_0(X)$.
\end{lemma}
\begin{proof}
By hypothesis, there exists a sequence $\{ S_n\} _n\in \chM _X$ with
 mean curvatures $H_n\searrow h_0(X)$ and unbounded areas. By item~(4) of
 Theorem~\ref{thm:index1}, the norms of the second fundamental forms of the
immersed $H_n$-spheres $S_n$ form a uniformly bounded sequence. A
standard compactness argument in elliptic theory implies that there
exists a pointed limit immersion $F\colon (\Sigma ,p)\la (X,e)$
which is a limit of left translations of restricted immersions of
compact domains in the $S_n$ into $X$, in the sense of
Definition~\ref{def6.1}.
\end{proof}

\begin{definition}
\label{def6.3} {\rm Suppose $F\colon (\Sigma ,p) \looparrowright
(X,e)$ is a pointed limit $H$-immersion of $\chM _X$ (in particular,
$\Sigma $ is not compact and possibility~(b) above occurs for $X$).
Let $\Delta(F)$ be the set of pointed immersions $F_1\colon (\Sigma
_1,p_1)\looparrowright (X,e)$ where $\Sigma _1$ is a complete,
non-compact connected surface, $p_1\in \Sigma _1$, $F_1(p_1)=e$ and
$F_1$ is obtained as a limit of $F$ under an (intrinsically)
divergent sequence of left translations. In other words, there exist
compact domains $\Omega _n\subset \Sigma $ and points $q_n\in \Omega
_n $ diverging to infinity in $\Sigma $ such that the sequence of
left translated immersions $\{ F(q_n)^{-1}\, F|_{\Omega _n}\} _n$
converges on compact sets of $\Sigma _1$ to $F_1$ as $n\to \infty $.
}
\end{definition}

\begin{corollary} \label{rank:limits}
Let $F\colon (\Sigma',p') \looparrowright (X,e)$ be a pointed limit
immersion of $\chM_X$ of constant mean curvature $H$. Then, the
space $\Delta(F)$ is non-empty and every $[f\colon (\Sigma ,p)
\looparrowright (X,e)]\in \Delta (F)$ is a limit surface of $\chM_X$
which is stable. Furthermore, if the $H$-potential for $X$ never
vanishes and $[f\colon (\Sigma ,p) \looparrowright (X,e)]\in \Delta
(F)$, then:
\begin{enumerate}[(1)]
\item There exists a non-zero
right invariant vector field $K_{\Sigma }$ on $X$ which is
everywhere tangent to $f(\Sigma)$.
\item $f(\Sigma)$ is topologically an immersed plane, annulus or  torus in $X$.
\item  $\Sigma$ is  diffeomorphic to a plane or
an annulus.
\end{enumerate}
\end{corollary}
\begin{proof}
The property that $\Delta(F)$ is non-empty follows from the facts
that $\Sigma '$ is non-compact and $F\colon \Sigma '\la X$ has
bounded second fundamental form. Moreover, a standard diagonal
argument shows that every element in $\Delta (F)$ also is a limit
surface of $\chM_X$.

We next check that if $[f\colon (\Sigma ,p) \looparrowright
(X,e)]\in \Delta (F)$, then $f$ is stable. Note that since the
spheres in $\chM_X$ have index one, then the index of $F$ is at most
one (otherwise $\Sigma '$ contains at least two compact disjoint
strictly unstable subdomains, which implies that the $H_k$-spheres
$S_{H_k}\in \chM_X$ that limit to the immersion $F$ also have two
compact, disjoint, smooth, strictly unstable subdomains, which
contradicts the index-one property of $S_{H_k}$). Arguing by
contradiction, if $f$ is not stable, then there exists a smooth
compact domain $D$ on $\Sigma$ that is strictly unstable. Since $f$
is a limit of the left translated immersions $F(q_n)^{-1}\,
F|_{\Omega _n}$ where the $\Omega _n\subset \Sigma '$ are compact
domains and $q_n\in \Omega _n$ is a divergent sequence in $\Sigma
'$, then $D$ must be the limit under left translation by
$F(q_n)^{-1}$ of a sequence of compact domains $D_n\subset \Omega
_n$. Since the points $q_n$ diverge in $\Sigma '$, then after
extracting a subsequence, we can assume that the $D_n$ are pairwise
disjoint. As $D$ is strictly unstable for $f$, then for $n$ large
$D_n$ is also strictly unstable for $F$. This implies that $\Sigma'$
contains two smooth, closed, disjoint strictly unstable domains, a
property that we have seen is impossible. Hence, the immersion $f$
is stable.

In the sequel, we will assume that the $H$-potential for $X$ never
vanishes. Next we show that $f(\Sigma )$ is everywhere tangent to a
right invariant vector field. By Theorem~\ref{rep}, the left
invariant Gauss map $G\colon \Sigma \to \esf^2\subset T_eX$ of $f$
is never of rank zero. Since $f$ is a limit surface of $\chM_X$ and
of $F$, then the differential of $G$ is never of rank two as we
explain next: otherwise, there exists a compact disk $D\subset
\Sigma $ whose image by $G$ contains a non-empty open disk in
$\esf^2$; since $D$ is compact, then $D$ is a smooth limit of
related pairwise disjoint disks $D_n$ in the domain $\Sigma '$ of
$F$. Each such $D_n$ is in turn a smooth limit of disks on spheres
in $\chM _X$, which easily contradicts the injectivity of the left
invariant Gauss maps of the spheres in $\chM_X$  that limit to $F$.
Hence, the rank of $G$ is 1 everywhere on $\Sigma $, and by
Corollary~\ref{invacor}, there exists a right invariant vector field
$K_{\Sigma}$ on $X$ whose restriction to $f(\Sigma)$ is tangent to
$f(\Sigma )$. Now item~(1) of the corollary is proved.

To prove item~(2), first suppose that $X$ is not isomorphic to $\su
$. Consider a non-zero right invariant vector field $K_{\Sigma }$ on
$X$ which is everywhere tangent to $f(\Sigma)$. Let $\G \subset X$
be the 1-parameter subgroup obtained under exponentiation of the
tangent vector $K_{\Sigma }(e)\in T_eX$. Since $\G $ is the integral
curve of $K_{\Sigma }$
 passing through $e$ and $K_{\Sigma }$ is right invariant, then the integral curve
of $K_{\Sigma }$ passing through any point $x\in X$ is the right
translation $\G x$ of $\G $ by $x$. Consider the quotient surface
$\G \setminus X =\{ \G x\ | \ x\in X\} $ whose points are the right
cosets of $\G $ and let $\Pi \colon X\to \G \setminus X$ be the
corresponding submersion. From the classification of the possible
Lie group structures on $X$, a simple case by case study shows that
there exists a properly embedded plane in $X$ that intersects each
right coset of $\G$ transversely in a single point. It follows that
$\Pi $ is a trivial $\R$-bundle over the topological plane $\G
\setminus X$. Thus, $f$ is a trivial $\R $-bundle over an immersed
curve $\a $ in $\G \setminus X$. This implies that $f(\Sigma) $ is a
complete immersed plane or annulus with a natural product structure
of $\R $ (the parameter of $\R $ identifies with the parameter along
the fibers of $\Pi $ contained in $f(\Sigma )$) with either $\R $
(if $\a $ is an immersed open arc) or possibly over $\esf^1$ (if $\a
$ is an immersed closed curve), which proves item~(2) when $X$ is
not isomorphic to $\su$. In this case, it follows that $\Sigma $ is
topologically a plane or an annulus; note that $f$ might factor
through a quotient immersion of an annulus and so $\Sigma $ might be
topologically a plane even if $\a$ is an immersed closed curve.

Finally, assume $X$ is isomorphic to $\su $. In this case,
$\Pi\colon X \to \G \setminus X$ has the structure of a smooth
oriented $\esf^1$-bundle of $X$ over $\S^2$ (this fiber
$\esf^1$-bundle can be identified, up to ambient orientation, with
the usual oriented Hopf fibration of $\S^3$ over $\S^2$, a fact that
we will not use here). A similar argument as the one given above
proves that $f(\Sigma) $  is an immersed annulus or torus. Since
$\Sigma $ is non-compact, then $\Sigma$ is diffeomorphic to a plane
or an annulus in this case. This completes the proofs of items (2)
and (3).
\end{proof}

\section{The proof of Theorem~\ref{main}.}
\label{sec:SU(2)}

Let $X$ be $\su$ with a left invariant metric. We start by proving
the surjectivity of the map $\cH\colon\mathcal{C}\to \R$ described
in Theorem~\ref{thm:index1}, or equivalently, that possibility (a)
at the beginning of Section~\ref{sec:limits} holds for $X$. To prove
this we argue by contradiction, so assume that possibility (b)
defined at the beginning of Section~\ref{sec:limits} occurs. By
Lemma~\ref{DeltaF} and Corollary~\ref{rank:limits}, there exist a
stable pointed limit $h_0(X)$-immersion $f\colon (\Sigma ,p)\la
(X,e)$ of $\chM_X$ and a right invariant vector field $K_{\Sigma }$
on $X$ which is everywhere tangent to $f(\Sigma )$, where $h_0(X)\in
[0,\infty)$ is defined in the statement of item~(5) of
Theorem~\ref{thm:index1}. If we consider $\su$ with its metric of
constant sectional curvature 1, then the integral curves of any
right invariant vector field are closed geodesics of length  $2\pi
$. Since the left invariant metric on $X $ is quasi-isometric to the
constant curvature 1 metric on $\su $ (via the identity map), then
the integral curves of $K_{\Sigma }$ are embedded closed curves of
uniformly bounded length in $X$. Since $f(\Sigma )$ is foliated by
integral curves of $K_{\Sigma }$, then $f(\Sigma )$ is an immersed
annulus or torus by Corollary~\ref{rank:limits}.  As $f$ is stable
and $f\colon \Sigma \to f(\Sigma )$ can be viewed as a trivial, a
$\Z$ or a $(\Z \times \Z )$-regular cover of the image annulus or
torus, then we conclude by Proposition~2.5 in~\cite{mpr19} that
$f(\Sigma )$ is a stable $h_0(X)$-surface.

Suppose for the moment that $f(\Sigma )$ is an immersed torus.
Consider a right invariant vector field $V$ on $\su$ such that
$V(e)$ is tangent to $f$ at $e$ and linearly independent with
$K_{\Sigma }(e)$. Since $\su$ has no two-dimensional subgroups and
$K_{\Sigma }$ is everywhere tangent to $f(\Sigma )$, then $V$ is not
everywhere tangent to $f(\Sigma )$. Let $N$ be a unit normal vector
field along $f$.
 Hence, $J=\langle V,N \rangle$ is a bounded
Jacobi function on the torus $f(\Sigma) $ that changes sign, which
contradicts stability.

Next assume that $\hat{\Sigma}:=f(\Sigma )$ is an immersed annulus.
We claim that the area growth of $\hat{\Sigma}$ is at most linear.
Let $\alfa:\R\flecha \hat{\Sigma}$ be a unit speed length minimizing
geodesic joining the two ends of $\hat{\Sigma}$. Recall that
$\hat{\Sigma}$ is foliated by integral curves of $K_{\Sigma}$ of
uniformly bounded length. Let $\mathcal{B}$ be the family of such
integral curves in $\hat{\Sigma}$ and let $R_1= {\rm sup}_{\beta\in
\mathcal{B}} \{ {\rm length} (\beta)\}<\infty $.

For each $R>R_1$ and $t\in\R$ we define

\begin{equation} \label{eqanuno}
C(t,R)= \bigcup \{ \beta\in \mathcal{B} \mid \beta \cap \alfa
([t-R,t+R]) \neq \mbox{\O}\}.
\end{equation}
The triangle inequality implies that for $R>R_1$, we have
\begin{equation} \label{equandos}
B_{\hat{\Sigma}}(\alfa(t),R)\subset C(t,2R)\subset
B_{\hat{\Sigma}}(\alfa(t),4R),
\end{equation}
where $B_{\hat{\Sigma }}(q,r)$ denotes the intrinsic ball in
$\hat{\Sigma }$ of radius $r>0$ centered at a point $q\in
\hat{\Sigma }$. Since  $f$ is a limit of immersions with uniformly
bounded second fundamental form, then $\hat{\Sigma}=f(\Sigma )$ has
bounded second fundamental form, and so we deduce by the Gauss
equation for $f$ that the Gaussian curvature of $\hat{\Sigma}$ is
bounded from below. Under these hypotheses, Bishop's second theorem
(see e.g., Theorem III.4.4 in Chavel~\cite{ch2}) gives that for
every $R_0>0$, there exists a constant $C=C(R_0)>0$ such that every
intrinsic ball in $\hat{\Sigma}$ of radius $R_0$ has area less than
$C$. So, by the second inclusion in equation \eqref{equandos} we see
that for every $R>R_1$ there exists a constant $K(R)$ so that

\begin{equation}\label{equantres}
{\rm Area} (C(t,2R)) \leq K(R)\hspace{1cm} \text{ for every $t\in
\R$}.
\end{equation}
Now observe that as $\a $ is parameterized by arc length, the
definition of $C(t,R)$ in \eqref{eqanuno} implies that
\begin{equation}
\label{equan4} C(0,R) \subset \bigcup_{j=-(n+1)}^{n+1}
C(j,2(R_1+1)),
\end{equation}
where $n$ is the integer part of $R$, $n=[R]$. Now,
\[
 \mbox{Area}(C(0,R))\stackrel{(\ref{equan4})}{\leq }\sum _{j=-(n+1)}^{n+1}\mbox{Area}(C(j,2(R_1+1)))
\]
\[\stackrel{(\ref{equantres})}{\leq }
 \sum _{j=-(n+1)}^{n+1}K(R_1+1)=(2n+3)K(R_1+1)<(2R+5)K(R_1+1)
\]
Since $B_{\hat{\Sigma}}(\alfa(0),R)\subset C(0,2R)$ by
(\ref{equandos}), it follows that $\hat{\Sigma}$ has at most linear
area growth, as claimed.

Since $\hat{\Sigma}$ has at most linear area growth, the conformal
structure of $\hat{\Sigma} $ is parabolic (see e.g.,
Grigor{'}yan~\cite{gri1}). From this point we can finish with an
argument similar to the one when $f(\Sigma)$ was a torus. Choose a
right invariant vector field $V$ of $X$ that is tangent to
$f(\Sigma)$ at $e$ but is not everywhere tangent to $f(\Sigma)$. If
$N$ stands for a unit normal vector field along $f$, then $J=\langle
V,N \rangle$ is a bounded Jacobi function on $\hat{\Sigma }$ that
changes sign, which contradicts the main theorem in Manzano, P\'erez
and Rodr\'\i guez~\cite{mper1}.

The above two contradictions for the possible topologies of
$\wh{\Sigma }$ show that the map $\cH\colon\mathcal{C}\to \R$
described in item~(5) of Theorem~\ref{thm:index1} is surjective.
Then, Theorem~\ref{thm:index1} implies that for all $H\in \R$ there
exists a sphere of constant mean curvature $H$ in $X$, unique up to
left translations in $X$. That is, items~(\ref{exist}) and
(\ref{uniq}) of Theorem~\ref{main} hold. Moreover, also by
Theorem~\ref{thm:index1} we see that all these constant mean
curvature spheres have index one, nullity three, and satisfy that
their Gauss maps are diffeomorphisms. Therefore, item~(\ref{index1})
of Theorem~\ref{main} holds. Note that item~(\ref{alex-embed})
follows from Corollary~\ref{alex-cor}.

We next explain how to define the center of symmetry of an immersed
$H$-sphere $f\colon S_H\la X$. Assume for the moment that $H\geq 0$.
By the already proven item~(\ref{alex-embed}) of Theorem~\ref{main},
the immersion $f\colon S_H\looparrowright X$ can be extended to an
isometric immersion $F\colon B\to X$ of a Riemannian three-ball such
that $\partial B=S_H$ is mean convex and $F|_{\partial B}=f$. The
choice of this isometric immersion $F$ is clear if the mean
curvature of $f$ is not zero. If $f$ is minimal, then by
Remark~\ref{rem44} we can choose $F\colon B\to X$ so that the
orientation of $S_H$ is the orientation induced as the boundary of
$B$.

After fixing the orthonormal basis $\{ (E_1)_e,(E_2)_e,(E_3)_e\}$ of
$T_eX$, we can consider the isotropy group $I_e(X)$ of the identity
element of $X$ in the isometry group of $X$ to be a subgroup of
$O(3)$, and the corresponding subgroup $I_e^+(X)$ of orientation
preserving isometries to be a subgroup of $SO(3)$; this
identification of $I_e(X),I_e^+(X)$ with subgroups of $O(3)$ is just
by the related linear action of the differentials of their elements
on $T_eX=\rth$. Moreover, every $\phi \in I_e^+(X)$ is an inner
automorphism of $X$,
\[
\phi (x)=a^{-1}xa,\quad x\in X,
\]
where $a$ is some element in the 1-parameter subgroup $\exp
(\mbox{Span}\{ v\} )$ of $X$, and  where $v$ is a unitary vector in
$T_eX$ such that $d\phi _e(v)=v$. In particular, each $\phi \in
I_e^+(X)$ is a group isomorphism of $\su$. Furthermore, $I_e^+(X)$
contains a $\Z_2\times \Z_2$ subgroup corresponding to the $\pi
$-rotations that appear in Proposition~\ref{propos3.2}; more
specifically (see Proposition~2.21 of~\cite{mpe11} for details).
\begin{ddescription} \label{des} The isotropy subgroup  $I_e(X)$
of  $e\in X$ in the isometry group of $X$ has the following
description: \ben
  \item $I_e^+(X)=I_e(X)=\Z_2\times \Z_2$ when the isometry group
of $X$ is three-dimensional.
\item If {\rm Iso}$(X)$ has dimension four, then $I_e^+(X)=I_e(X)$
is an $\esf^1$-extension of $\Z_2$ corresponding to the rotations
about the 1-parameter subgroup of $X$ generated by the Ricci
eigendirection with multiplicity one; for concreteness in this case
we may assume that this Ricci eigendirection with multiplicity one
is $E_3$.
\item When {\rm Iso}$(X)$ has dimension six, then $X$ has positive constant sectional
curvature and $I_e^+(X)=SO(3)$, $I_e(X)=O(3)$. \een
\end{ddescription}

We next define injective group homomorphisms from $I_e^+(X)$ into
the groups $\mbox{Iso}^+(S_H)$, $\mbox{Iso}^+(B)$ of orientation
preserving isometries of $S_H$ and $B$, respectively. Given $\phi
\in I_e^+(X)$ different from the identity element, let $v\in T_eX$
be a unitary vector such that $d\phi _e(v)=v$ ($v$ is unique up to
sign). As the left invariant Gauss map $G$ of $S_H$ is a
diffeomorphism, there is a unique $p_v\in S_H$ such that $G(p_v)=v$.
Consider the left translated sphere $(l_{f(p_v)^{-1}}\circ f)(S_H)$.
By Remark~\ref{fixed-point}, as $\phi $ is an isometry of $X$ that
leaves $e$ invariant and $d\phi _{e}(v)=v$, then $\phi $ leaves
$(l_{f(p_v)^{-1}}\circ f)(S_H)$ invariant. In particular, $\phi $
induces an orientation preserving isometry $\wh{\phi}\colon S_H\to
S_H$ such that
\[
l_{f(p_v)}\circ \phi \circ l_{f(p_v)^{-1}}\circ f=f\circ \wh{\phi }.
\]
Observe that $d\phi _e\circ G=G\circ \wh{\phi }$. This implies that
we have an injective group homomorphism
\[
 \aleph \colon I_e^+(X) (\subset SO(3))\to \mbox{Iso}^+(S_H), \quad \aleph (\phi )=\wh{\phi }.
\]
Note that $\wh{\phi }$ can be extended uniquely to a small
neighborhood of every point of $S_H$ in $B$ (because the submersion
$F\colon B\to X$ is a local isometry and $\phi \colon X\to X$ is an
isometry), and since $S_H$ is compact, then $\wh{\phi }$ can be
extended to an isometry of a fixed size neighborhood $S_H(\ve )$ of
$S_H$ in $B$, an extension that we will denote by $\wt{\phi }\colon
S_H(\ve )\to S_H(\ve )$. As $B$ is simply connected, then a standard
monodromy argument ensures that we can extend $\wt{\phi }$ to an
isometry $\wt{\phi }\colon B\to B$. In particular,
\begin{equation}
  \label{eq:cs1}
  l_{f(p_v)}\circ \phi \circ l_{f(p_v)^{-1}}\circ F=F\circ \wt{\phi }\quad \mbox{ and }
\quad  \wt{\phi }|_{S_H}=\wh{\phi }.
\end{equation}
This construction defines an injective group homomorphism
\[
 \Psi \colon I_e^+(X) (\subset SO(3))\to \mbox{Iso}^+(B), \quad \Psi (\phi )=\wt{\phi }.
\]

By Theorem~3 in~\cite{my1},  every compact Lie subgroup $T$ of the
group of orientation preserving diffeomorphisms Diff$^+(\B)$ of the
closed unit ball $\B\subset \rth$ centered at the origin  is
conjugate to the action of a compact Lie subgroup $\wh{T}$ of
$SO(3)\subset\mbox{\rm Diff}^+(\B)$, as long as $T$ is not
isomorphic to the alternating group $A_5$ (this is the group of
orientation preserving isometries of a regular icosahedron). In
particular, since by Description~\ref{des},
 $\Psi ( I_e^+(X))\subset\mbox{\rm Diff}^+(B)$ is not isomorphic to $A_5$,
then $B$  has a point $p_0$ that is fixed by all of the elements of
$\Psi ( I_e^+(X))$. Since $\Psi ( I_e^+(X))$ is not isomorphic to
$\S^1$ or to a finite cyclic group, $p_0$ is the unique such fixed
point.

We now define the center of symmetry of the immersed $H$-sphere
$S_H$ with $H\geq 0$ to be $F(p_0)$. Next we prove that every
isometry $\psi \colon X\to X$ that fixes $F(p_0)$ leaves $f(S_H)$
invariant. To do this, we will first define an isomorphism from
Iso$^+(B)$ to $I^+_{F(p_0)}(X)$, where $I^+_{F(p_0)}(X)$ denotes the
set of orientation preserving isometries of $X$ that fix $F(p_0)$.
Given $\wt{\phi }\in \mbox{Iso}^+(B)$, consider a geodesic ball
$B(p_0,\de )\subset B$ of radius $\de >0$ small enough so that
$F|_{B(p_0,\de )}\colon B(p_0,\de )\to F(B(p_0,\de ))$ is an
isometry. Then, $(F|_{B(p_0,\de )})\circ \wt{\phi }\circ
(F|_{B(p_0,\de )})^{-1}$ extends uniquely to an isometry $\Lambda
(\wt{\phi })\in I^+_{F(p_0)}(X)$, and the map $\Lambda \colon
\mbox{Iso}^+(B)\to I^+_{F(p_0)}(X)$ is clearly a continuous
injective homomorphism. Also, $\Lambda (\wt{\phi })$ leaves $f(S_H)$
invariant for every $\wt{\phi }\in \mbox{Iso}^+(B )$. We next check
that $\Lambda $ is onto (and hence, it is an isomorphism). In case
(1) of Description~\ref{des}, $I_{F(p_0)}^+(X)$ is finite so
$\Lambda $ is an isomorphism. In case~(2), we have a composition of
injective Lie group homomorphisms
\[
I_e^+(X)\stackrel{\Psi }{\to }\mbox{Iso}^+(B)\stackrel{\Lambda }{\to
}I^+_{F(p_0)}(X),
\]
and the groups at the extrema of the last sequence are isomorphic;
in fact, each of $I_e^+(X)$, $I^+_{F(p_0)}(X)$ has two components
(both homeomorphic to $\esf^1$), from where we deduce that both
$\Psi $ and $\Lambda $ are homomeorphisms. Finally, in case (3) of
Description~\ref{des} we have $I_e^+(X)=SO(3)$, $X$ has constant
sectional curvature and $f(S_H)$ is the orbit of any of its points
by the group of orientation preserving isometries of $X$ that fix
$F(p_0)$; in particular, $f(S_H)$ is in this case a constant
distance sphere from $F(p_0)$. Therefore, in any case every isometry
$\psi \in I^+_{F(p_0)}(X)$ leaves $f(S_H)$ invariant. Observe that
if $X$ contains an orientation reversing isometry, then $X$ has
constant sectional curvature and thus every isometry of $X$ fixing
$F(p_0)$ leaves $f(S_H)$ invariant.   This completes the proof of
item~(\ref{familySt}) of Theorem~\ref{main} in the case that $S_H$
has non-negative mean curvature.

If $S_H$ is an immersed $H$-sphere with $H<0$, then we could change
the orientation on $S_H$ and then apply the above procedure to
define its center of symmetry. This approach is not satisfactory,
since it produces a jump discontinuity on the center of symmetry
when passing from the case $H\geq 0$ to the case $H<0$. To avoid
this discontinuity we argue as follows. First note that the three
geodesics $\G _1,\G_2,\G_3$ that appear in
Proposition~\ref{propos3.2} intersect at $e=I$ and $-I$ (here we
have identified $X$ with the set of unit length quaternions), and
that the isotropy groups of $I,-I$ in the isometry group of $X$
coincide. Also observe that given $x\in X$, the isotropy group
$I_x(X)$ is conjugate to $I_e(X)$, namely $I_x(X)=\{ l_x\circ \phi
\circ l_{x^{-1}}\ | \ \phi \in I_e(X)\} $. Therefore,
$I_x(X)=I_{A(x)}(X)$ where $A\colon X\to X$ is the antipodal map,
and if $I_x(X)=I_y(X)$ for $x,y\in X$, $x\neq y$, then $y=A(x)$.
With this in mind, we define the center of symmetry of an immersed
sphere $f\colon S_H\la X$ with $H<0$ as $A(F(x_0))\in X$, where
$F\colon B\to X$ is the isometric immersion of a three-ball that
extends $f$ endowed with the opposite orientation on $S_H$ (that is,
$f\colon S_H\la X$ is considered to be a $(-H)$-sphere). With this
definition, item~(\ref{familySt}) of Theorem~\ref{main} holds
trivially for every $H$-sphere in $X$, $H\in \R $.

In order to prove the final statement of Theorem~\ref{main}, we will
define a natural normalization for the $H$-spheres in $X$. Let $\{
f_H\colon S_H\la X\ | \ H\in \R\} $
 be the family of immersed spheres of constant mean curvature $H$ in $X$ with $H\in \R$ varying,
 normalized so that for
each element $f_H\colon S_H\la X$ of this family,
 the point $q_H\in S_H$  where the left invariant Gauss map of $f_H$
 takes the value $(E_3)_e\in \esf^2\subset T_eX$
satisfies $f_H(q_H)=e$.
By Step 3 in the proof of Theorem~\ref{thm:index1}, this is an
analytic family with respect to the mean curvature parameter $H$.
Let $q_H^*\in S_H $ be the unique point where the left invariant
Gauss map of $S_H$ takes the value $-(E_3)_e$, and observe that
since $f_H(S_H)$ is invariant under the $\pi $-rotation about
$\G_3$, then $f_H(q_H^*)\in \G_3$. We parameterize the subgroup
$\G_3$ by arc length by the Lie group exponential
\[
\pi(t)=\exp(t (E_3)_e)\colon \R\flecha \G_3.
\]
Then, the map $H\in \R\mapsto  f_H(q_H^*)\in \G_3$ defines an
analytic function $t\colon  \R\to \R$ by $\pi(t(H))=f_H(q_H^*)$. Now
consider the point $x_H=\pi(t(H)/2)$. For $H$ large and positive,
the sphere $S_H$ is embedded and its center of symmetry is the point
$x_H$, since $x_H$ is the midpoint of the short arc of $\G_3$ with
endpoints $e$ and $f_H(q_H^*)$. Now define $\hat{f}_H =l_{x_H^{-1}}
\circ f_H\colon S_H\la X $. Then, $\{ \hat{f}_H\colon S_H\la X\ | \
H\in \R\} $ is an analytic family with respect to the parameter $H$.
If $H$ is large and positive, then by definition, we see that the
center of symmetry of $\hat{f}_H$ is $e$. Therefore,  we conclude by
analyticity that all the oriented spheres $\hat{f}_H\colon S_H\la X$
have $e$ as their center of symmetry. This concludes the proof of
Theorem~\ref{main}.

\section{The geometry of minimal spheres.}
\label{sec:SU(2)}

In this section we apply Theorems~\ref{main} and~\ref{thm:index1} to
describe in detail the geometry of the unique minimal sphere in an
arbitrary compact, simply connected homogeneous three-manifold $X$,
where we identify $X$ with $\su$ endowed with a left invariant
metric. It follows immediately from the proof of the next theorem
that when the isometry group of such an $X$ has dimension greater
than three, then the minimal sphere $S(0)$ described in the last
statement of Theorem~\ref{main} is the set of elements in $\su$ of
order four, where we choose the base point $e$ of $X$ to be the
identity matrix in $\su$; this last property also follows from
Remark 5 in Torralbo~\cite{tor2}.

\begin{theorem}
\label{embed:su2} Suppose $X$ is $\su$ endowed with a left invariant
metric. Let $e=I$ be the base point of $X$, where $I$ is  the
identity matrix in $\su$. If $S(0)$ is the immersed minimal sphere
in $X$ given in Theorem~\ref{main}, then:
\begin{enumerate}[(1)]
\item $S(0)$ is embedded and separates $X$ into two isometric balls that are interchanged
under left multiplication $l_{-I}$ by $-I$.
\item There exist three geodesics $\a _1,\a _2,\a _3$ of $X$ that intersect orthogonally in pairs,
such that each $\a _i$ is contained in $S(0)$, $\a _i$ is invariant
under the left action of a 1-parameter subgroup $\G_i$  of $X$ and
$\a _i$ is the fixed point set of an order-two, orientation
preserving isometry $\psi _i\colon X\to X$.
\item Consider the group $\mbox{\rm Iso}_X(S(0))$
of isometries of $X$ that leave $S(0)$ invariant. Then, $\mbox{\rm
Iso}_X(S(0))$ contains the subgroup
\[
\Delta =\{ l_I,l_{-I},\psi _i,\phi _i=(l_{-I})\circ \psi _i\ | \
i=1,2,3\},
\]
which is isomorphic to $\Z_2\times \Z_2\times \Z_2$ (here, $\phi
_1,\phi _2,\phi _3$ are defined in Proposition~\ref{propos3.2}).
Furthermore, if the isometry group of $X$ is three-dimensional, then
$\mbox{\rm Iso}_X(S(0))=\Delta $.
\end{enumerate}
\end{theorem}
\begin{proof}
Consider the minimal sphere $S(0)$ given in the family $S(t)$
described at the end of Theorem~\ref{main}. Since $S(0)$ lies in the
family $\mathcal{C}$ defined in item~(3) of
Theorem~\ref{thm:index1}, then $S(0)$ is embedded (see the proof of
Corollary \ref{alex-cor}).

As explained in Section~\ref{sec:background}, if $E_1,E_2,E_3$ is an
orthonormal canonical left invariant frame for $X$, then the Ricci
tensor of $X$ diagonalizes in this basis, and for $i=1,2,3$, the
corresponding 1-parameter subgroup $\G _i$ of $X$ with $\G
_i'(0)=(E_i)_I$ is a geodesic of $X$ and it is the fixed point set
of the rotation $\phi _i$ of angle $\pi $ about $\G _i$. Let $D_i$
be an immersed minimal disk in $\su $ with boundary $\G_i $. For
instance, $D_i$ could be chosen to be a solution of the classical
Plateau problem for $\G_i$ (i.e., a disk of least area among all
disks bounded by $\G _i$), which exists by results of
Morrey~\cite{mor1}; this solution is free of branch points in its
interior by Osserman~\cite{os7} and free of boundary branch points
by Gulliver and Lesley~\cite{gl2} since both $\G_i$ and the ambient
metric are analytic. After $\pi $-rotation of $D_i$ about $\G_i  $,
$D_i\cup \phi_i(D_i)$ is an immersed minimal sphere in $\su $ (see
for instance Theorem~1.4 in Harvey and Lawson~\cite{hl1} for this
Schwarz reflection principle in the general manifold setting), which
by item~(2) of Theorem~\ref{thm:index1} is equal to some left
translate of $S(0)$, say
\[
S(0)=l_{b_i}\left( D_i\cup \phi _i(D_i)\right)
\]
for some $b_i\in \su $.

Define $\alfa_i=l_{b_i} (\Gamma_i)$. By the previous discussion,
$\alfa_i$ is a geodesic of $X$ contained in $S(0)$ and $\a _i$ is an
integral curve of the left invariant vector field $E_i$.
Furthermore, the diffeomorphism
\[
\psi_i:=l_{b_i}\circ \phi_i \circ l_{b_i^{-1}}\colon X\to X
\]
is an isometry of $X$, namely the $\pi $-rotation around $\alfa_i$,
and $\psi _i$ leaves $S(0)$ invariant.

As each $\alfa_i$ is contained in $S(0)$ and it is an integral curve
of the left invariant vector field $E_i$, then the left invariant
Gauss map $G$ of $S(0)$ satisfies that $G(\alfa_i)\subset \S^2$ is
the great circle in $\S^2$ of unit vectors in $T_I X$ orthogonal to
$(E_i)_I$. Since $G$ is a diffeomorphism, there exist unique points
$A_1,A_1^*,A_2,A_2^*,A_3,A_3^* \in S(0)$ such that:
\begin{enumerate}
\item $\alfa_1\cap \alfa_3=\{A_2,A_2^*\}$, $\alfa_2\cap \alfa_3=\{A_1,A_1^*\}$, $\alfa_1\cap \alfa_2=\{A_3,A_3^*\}$.
\item  The left invariant Gauss map of $S(0)$ at $A_i$ (resp. at $A_i^*$) is $(E_i)_I$ (resp. $(-E_i)_I$).
\end{enumerate}
Observe that this proves item~(2) of Theorem~\ref{embed:su2}.

Recall that $S(0)$ is invariant under all isometries of $X$ that fix
$I$. Thus, $\phi_i$ leaves $S(0)$ invariant. As $\phi _i$ is the
$\pi $-rotation around $\G _i$, then $\G_i$ intersects $S(0)$
orthogonally. This implies that the tangent vectors to $\G _i$ at
the intersection points in $\G _i\cap S(0)$ are normal to $S(0)$ at
these points. As $\G _i$ is an integral curve of $E_i$ and as the
left invariant Gauss map of $S(0)$ is a diffeomorphism, then we
conclude that $\Gamma_i\cap S(0)=\{A_i,A_i^*\}$ for $i=1,2,3$.

We claim that $A_i^*= l_{-I} (A_i)$ for $i=1,2,3$ (this explains the
notation $A_i^*=-A_i$ in Figure~\ref{figS(0)}), and that $A_i$ is
the midpoint in $\G_i$ between $I$ and $-I$. We will prove it for
$i=1$; the proofs of the other two cases are analogous. As $\phi_2$
leaves $S(0)$ invariant and also leaves $\G _1$ invariant (by
Proposition~\ref{propos3.2}), then $\phi_2(A_1)=A_1^*$ and $\phi_2
(A_1^*)=A_1$. As $\phi_2(I)=I$ and $\phi_2$ is an isometry, we
deduce that $A_1$ and $A_1^*$ are at the same distance to $I$ along
$\G_1$. By Proposition~\ref{propos3.2}, the $\pi$-rotation $\psi_2$
around $\alfa_2$ maps integral curves of $E_i$ to integral curves of
$E_i$ for each $i=1,2,3$. As $\psi_2$ fixes $\alfa_2$ and $\G_1\cap
\alfa_2=\{A_1,A_1^*\}\neq \mbox{\O}$, we conclude that $\psi_2$
sends $\G_1$ to $\G_1$. As $A_1,A_1^*$ are fixed under $\psi_2$, we
see then that $\psi_2$ maps the segment in $\G_1$ with endpoints
$A_1,A_1^*$ that contains $I$ as midpoint into the segment in $\G_1$
with endpoints $A_1,A_1^*$ that contains $-I$ as midpoint. As
$\psi_2$ is an isometry, the two segments have equal length. This
property implies that $A_1^*= l_{-I} (A_1)$ as claimed.

\begin{figure}
\begin{center}
\includegraphics[height=5.1cm]{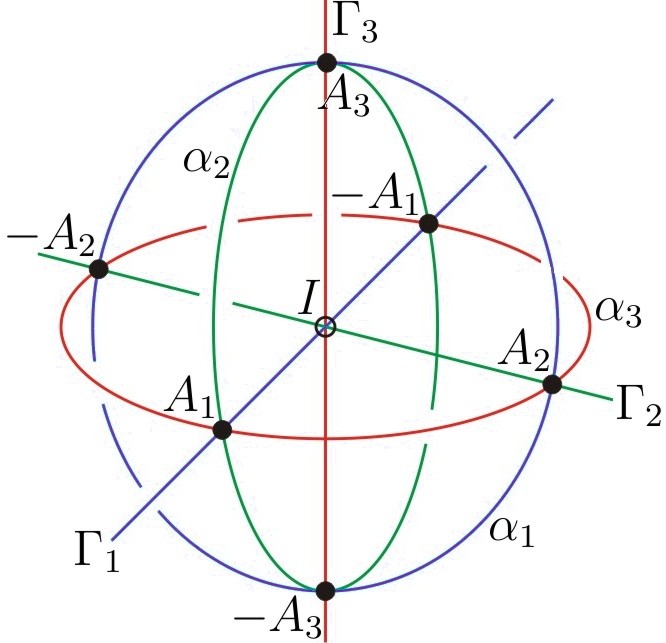}
\caption{The three straight lines passing through the identity
matrix $I\in \su $ represent the three 1-parameter subgroups $\G _i$
generated by an orthonormal canonical basis of $X$ (each of the
$\G_i$ passes through $-I$ which is represented by infinity in this
picture). The midpoints of each pair of arcs $\G _i-\{ \pm I\} $ are
the points $\{ \pm A_i\} $, and the left translation of each $\G_i$
by a point $\pm A_j$, $j\neq i$, is a geodesic $\a _i$ contained in
the minimal sphere $S(0)$, which is invariant under $\pi $-rotation
about $\a _i$, $i=1,2,3$.}
 \label{figS(0)}
\end{center}
\end{figure}

Observe that we have proved that
\begin{equation}
\label{eq:7.1} \a _1=A_2\, \G _1,\quad \a _2=A_1\G _2,\quad \a
_3=A_1\G _3,
\end{equation}
where $\{ A_i,-A_i=l_{-I} ( A_i)=A_i^*\} $ are the midpoints of the
two arcs in which the points $I,-I$ divide $\G_i$. In particular,
$A_i^{-1}=-A_i$ (inverse elements in the group structure of $\su $).
By (\ref{eq:7.1}), the $\pi$-rotations $\psi_i$ around $\a _i$ are
given by
\[
\psi _1=l_{A_2}\circ \phi _1\circ  l_{A_2}^{-1}, \quad \psi
_2=l_{A_1}\circ \phi _2 \circ l_{A_1}^{-1}, \quad \psi
_3=l_{A_1}\circ \phi _3 \circ l_{A_1}^{-1}.
\]
An algebraic calculation shows that $\psi_1\circ \psi_2\circ \psi_3=
l_{-I}$ (evaluation at the points $I, A_i$ show that the desired
equality holds at four linearly independent points of the unit
sphere in $\R^4$ equipped with its standard metric; now the equality
holds by elementary linear algebra). As each $\psi _i$ leaves $S(0)$
invariant, then we deduce that $S(0)$ is also invariant under
$l_{-I}$. Note that we already proved that $S(0)$ divides $\su$ into
two isometric components, of which one contains $I$ and the other
contains $-I$. Therefore, we have proved item~(1) of
Theorem~\ref{embed:su2}.

We next prove item~(3). By our previous discussion, we know that
$\{l_{I},\psi _i,\phi _i\ | \ i= 1,2,3\}\subset \mbox{Iso}_X(S(0))$.
Algebraic computations show that $l_{-I}\circ \psi_i =\phi_i$,
$i=1,2,3$, which proves that $\Delta \subset \mbox{Iso}_X(S(0))$,
where $\Delta =\{ l_{I},l_{-I},\psi _i,\phi _i=l_{-I}\circ \psi _i\
| \ i= 1,2,3\} $. This proves the first sentence in item~(3) of
Theorem~\ref{embed:su2} (actually, $\Delta ^*=\{ l_{I}, \phi _1,\phi
_2,\phi _3\} $ is the $(\Z_2\times \Z_2$)-subgroup of $\Delta $ of
isomorphisms of $\su $ and $\Delta = \Delta ^*\cup \{ l_{-I}\circ
\chi \ | \ \chi \in \Delta ^*\} $). In the case that the left
invariant metric on $X$ has three-dimensional isometry group, then
Proposition~2.24 in~\cite{mpe11} ensures that $X$  has no
orientation reversing isometries and that the isometries of $X$ that
preserve $I$ are precisely the ones in $\Delta^* $. This last
observation completes the proof of item~(3) of
Theorem~\ref{embed:su2}.
\end{proof}

\begin{remark}
\label{smith}
{\rm
Let $X$ be a metric Lie group isomorphic to $\su$ and let $Y$ be the related quotient
metric Lie group $Y=X/\{I,-I\}$,  which is isomorphic to   $\mathrm{SO}(3)$.  Since the
minimal sphere $S(0)$ in Theorem~\ref{embed:su2} is invariant under
left multiplication by $-I$, then
$\P(0):=S(0)/\{I,-I\}\subset Y$ is an embedded minimal projective plane and any other minimal
projective plane in $Y$ is congruent by an ambient isometry to $\P(0)$.
Furthermore, as $S(0)$ has index one, then $\P(0)$ is stable (this means that
$-\int _{S(0)}u\mathcal{L}u\geq 0$ for all $u\in C^{\infty }(S(0))$
which is anti-invariant with respect to the left multiplication by $-I$, where
$\mathcal{L}$ is the Jacobi operator of $S(0)$ defined in equation (\ref{eq:Jacobi})).
The existence of a minimal projective plane
in $Y$ that has least area among all embedded projective planes in $Y$ follows
from the main existence theorem in~\cite{msy1}; this second approach gives an
independent proof to the existence of an embedded minimal sphere
in $X$ that avoids using the main theorem in~\cite{smith1}.  In~\cite{mmp1}, we study the
classification problem of compact minimal surfaces
bounded by integral curves of left invariant vector fields in $X$, and as a consequence
of this study, we give an independent
proof of the existence of an embedded minimal sphere in $X$ based on
disjointness and embeddedness properties of  the Hardt-Simon
solution to the oriented Plateau's for these special curves that are described  in~\cite{hrs1}.
}\end{remark}

\bibliographystyle{plain}
\bibliography{bill}
\end{document}